\documentclass[reqno,11pt]{amsart}
\makeatletter
\def\section{\@startsection{section}{1}%
	\z@{.7\linespacing\@plus\linespacing}{.5\linespacing}%
	{\bfseries
		\centering
}}
\def\@secnumfont{\bfseries}
\makeatother

\usepackage{amsmath}
\usepackage{amsfonts}
\usepackage{amssymb}
\usepackage{graphicx}
\usepackage{xcolor}
\usepackage{soul}
\usepackage{ulem}
\usepackage{lineno}
\newtheorem{theorem}{Theorem}[section]

\newtheorem{corollary}[theorem]{Corollary}

\newtheorem{definition}[theorem]{Definition}

\newtheorem{proposition}[theorem]{Proposition}
\newtheorem{remark}[theorem]{Remark}

\numberwithin{equation}{section}
\usepackage{lmodern}
\usepackage[colorlinks=true, allcolors=blue]{hyperref}
\colorlet{blu1}{blue!70!black}
\colorlet{blu2}{blue!50!black}
\colorlet{blu3}{blue!70!red}
\colorlet{blu4}{blue!60!green}
\colorlet{red1}{red!80}
\colorlet{red2}{red!50!black}
\colorlet{red3}{red!70!yellow}
\colorlet{red4}{red!50!yellow}
\colorlet{yel1}{yellow!50!black}
\colorlet{yel3}{yellow!20!blue}
\colorlet{gre1}{green!60!blue}
\colorlet{gre2}{green!60!black}
\colorlet{gre3}{green!40!black}
\catcode `\@=12 \textwidth16cm \textheight23cm \hoffset-1cm   \voffset-2.0cm
\begin{document}
	
\begin{center}
{\bf\Large Some combinatorial aspects of $(q,2)-$Fock space} \vspace{1.5cm}
		
{\bf\large Yungang Lu}\vspace{0.6cm}
		
{Department of Mathematics, University of Bari ``Aldo Moro''}\vspace{0.2cm}
		
{Via E. Orabuna, 4, 70125, Bari, Italy}	\vspace{1cm}

e-mail:yungang.lu@uniba.it
\end{center}\vspace{1.5cm}
	
\centerline{\bf\large Abstract}\bigskip\noindent
We introduce the $(q,2)-$Fock space over a given Hilbert space, calculate the explicit form of a product of the creation and annihilation operators acting on the vacuum vector, demonstrate that this explicit form involves a specific subset of the set of all pair partitions, and provide a detailed characterization of this subset.
	
	

	\bigskip\bigskip

\section{Introduction}\label{DCT-Intro}\bigskip

In \cite{YGLu2022a}, we have studied a new type Fock space $\Gamma_{q,m}({\mathcal H})$: the $(q,m)-$Fock space over a given Hilbert space ${\mathcal H}$. The creation and annihilation operators on the $(q,m)-$Fock space $\Gamma_{q,m}({\mathcal H})$ verifies the following {\bf $(q,m)-$communi-cation relation}:
\begin{align*}
&A(f)A^+(g)-q{\bf p}_mA^+(g)A(f)=\langle f,g \rangle, \quad\forall f,g\in{\mathcal H}\\
&{\bf p}_k A(f)=A(f){\bf p}_{k+1}, \quad\forall f\in{\mathcal H} \text{ and } k\in{\mathbb Z}
\end{align*}
here, $q\in[-1,1]$ and $\{{\bf p}_k\}_{k\in {\mathbb Z}}$ is a {\it increasing family of projectors} on $\Gamma_{q,m}({\mathcal H})$.

Motivated by the $(q,m)-$communication relation mentioned above, we introduced in \cite{YGLu2022a} the Quan algebra, which is a non--trivial generalization of the conventional quan algebra ($q-$algebra for the simplicity), where, the standard $q-$commutation relation is replaced by the above formulated $(q,m)-$communication relation. The main contribution of \cite{YGLu2022a} was to explore some essential algebraic aspects of the Quan algebra, with a particular focus on providing a type of Wick's Theorem within the framework of the Quan algebra.

Our interest in the present paper is to see some important combinatorial aspects of the $(q,m)-$Fock space with a focus on the specific case of $m=2$, which is the simplest non--trivial case. In order to outline the main problem and main results, let's revisit some commonly used terminologies and some well--known related results.

Recall firstly that for any $n\in{\mathbb N}^*$, $\{(i_h, j_h)\}_ {h=1}^n$ is a {\bf pair partition} of the set $\{1,\ldots,2n\}$ if

$\bullet$ the set $\{i_h,j_h:h=1,\ldots,n\}$ is nothing else than $\{1,\ldots,2n\}$;

$\bullet$ $i_h<j_h$ for any $h\in\{1,\ldots,n\}$ and $1=i_1<\ldots<i_n<2n$.\\
Furthermore, a pair partition $\{(i_h,j_h)\}_ {h=1}^n$ is said to be {\bf non--crossing} if for any $1\le k<h\le n$, the following equivalence holds:
\begin{align}\label{CaCon19a1}i_k<i_h<j_k\ \iff\ i_k<j_h<j_k
\end{align}
Usually, $i_h$'s are referred to as the {\bf left indices}, and $j_h$'s as the {\bf right indices} of the pair partition $\{(i_h, j_h)\}_ {h=1}^n$.

For any $m,n\in{\mathbb N}^*$, it is common to denote:
\begin{align*}
PP(2n):=&\text{the totality of pair partitions of }\{1,\ldots,2n\} \notag\\
NCPP(2n):=&\text{the totality of non--crossing pair partitions of }\{1,\ldots,2n\}
\end{align*}
and
\begin{align*}
&\{-1,1\}^m:=\big\{\text{functions on }\{1,\ldots,m\} \text{ and valued in }\{-1,1\} \big\}\notag\\
&\{-1,1\}^{2n}_+:=\big\{\varepsilon\in \{-1,1\}^{2n} :\sum_{h=1}^{2n}\varepsilon(h) =0,\, \sum_{h=p}^{2n}\varepsilon(h)\ge 0,\ \forall p\in \{1,\ldots,2n\}\big\}  \notag\\
&\{-1,1\}^{2n}_-:=\{-1,1\}^{2n}\setminus \{-1,1\}^{2n}_+
\end{align*}
Moreover, for any $\varepsilon\in\{-1,1\}^{2n}$, the following notations are frequently employed:
\begin{align*}
&PP(2n,\varepsilon):=\big\{\{(i_h, j_h)\}_ {h=1}^n\in PP(2n):\,\varepsilon^{-1}(\{-1\}) =\{i_h:h\in\{1,\ldots,n\}\}\big\}\notag\\
&NCPP(2n,\varepsilon):=NCPP(2n)\cap PP(2n,\varepsilon)
\end{align*}
As shown in \cite{Ac-Lu96} and \cite{Ac-Lu2022a}, $PP(2n,\varepsilon)=\emptyset$ when $\varepsilon\in\{-1,1\}^{2n}_-$; however, in case $\varepsilon\in\{-1,1\}^{2n}_+$, hold the following:

$\bullet$ the set $NCPP(2n,\varepsilon)$  comprises exactly one element which will be denoted as $\{(l^\varepsilon_h, r^\varepsilon_h)\}_ {h=1}^n$ throughout this paper;

$\bullet$ any $\{(i_h, j_h)\}_ {h=1}^n\in PP(2n,\varepsilon)$ satisfies the following:
\[i_h=l^\varepsilon_h,\ \forall h\in\{1,\ldots,n\};\quad \big\{j_h:h\in\{1,\ldots,n\} \big\}=\big\{r^\varepsilon_h:h\in\{1,\ldots,n\}\big\}\]

$\bullet$ the cardinality of the set $PP(2n,\varepsilon)$ is equal to $\prod_{h=1}^n (2h-l^\varepsilon_h)$.

Pair partition and non--crossing pair partition
play a crucial role in calculation of the vacuum expectation of a product of creation and annihilation operators defined on a given Fock--space:
\begin{align}\label{DCT02c09}
\big\langle \Phi,A^{\varepsilon(1)}(f_1)\ldots A^{\varepsilon(2n)}(f_{2n})\Phi\big\rangle
\end{align}
with $n\in{\mathbb N}^*$ and $\varepsilon\in\{-1,1\}^{2n}_+$. Hereinafter,

$\bullet$ for any $f\in{\mathcal H}$, $A^+(f)$ (respectively, $A(f)$) is the creation (respectively, annihilation) operator with the test function $f$ respectively;

$\bullet$ for any $\varepsilon \in\{-1,1\}$ and $f\in\mathcal{H}$,
\begin{align*}
A^\varepsilon(f):=\begin{cases}A^+(f),&\text{ if } \varepsilon=1\\ A(f),&\text{ if } \varepsilon=-1
\end{cases}
\end{align*}

It is well known that the expression \eqref{DCT02c09}  depends on both pair partitions determined by $\varepsilon$ and the Fock structure. Often (but not always), the expression takes the form:
\begin{align*} \sum_{\{(l_h,r_h)\}_{h=1}^n\in {\mathcal P}_n(\varepsilon)}\alpha_n(\{(l_h,r_h)\}_{h=1}^n) \prod_{h=1}^n\langle f_{l_h},f_{r_h}\rangle
\end{align*}
Here, ${\mathcal P}_n(\varepsilon)$ represents a specific subset of $PP(2n,\varepsilon)$; $\alpha_n(\{(l_h,r_h)\} _{h=1}^n) \in{\mathbb C} \setminus\{0\}$ for any $\{(l_h,r_h)\}_{h=1}^n\in {\mathcal P}_n(\varepsilon)$.
For example,

$\bullet$ if $A^+$ and $A$ are creation and annihilation operator on the $q-$Fock space with $q\in[-1,1]$, which includes Bosonic, Fermionic and free Fock space, ${\mathcal P}_n(\varepsilon)$ consists of the pair partition $\{(l^\varepsilon_h, r^\varepsilon_h)\}_ {h=1}^n$ if $q=0$, and it is $PP(2n,\varepsilon)$ if $q\ne 0$.

$\bullet$ if $A^+$ and $A$ are creation and annihilation operator on the Boolean Fock space (see \cite{BGor-Schu2004} and references therein), ${\mathcal P}_n(\varepsilon)$ consists of the pair partition $\{(2h-1,2h)\}_{h=1}^n$ if $\varepsilon^{-1} (\{1\})= \big\{2h:\,h\in\{1,\ldots,n\}\big\}$; otherwise, it is empty.

The $(q,2)-$Fock space provides a new example where the set ${\mathcal P}_n(\varepsilon)$ is neither as large as $PP(2n,\varepsilon)$ nor as small as $PP(2n, \varepsilon)\cap NCPP(2n)$. Moreover, its relationship to Catalan's trapezoids is discussed in \cite{YGLu2022c}.

The main goal of this paper is to obtain the explicit form of ${\mathcal P}_n (\varepsilon)$. Moreover, we also demonstrate that ${\mathcal P}_n (\varepsilon)$ is determined uniquely by $\varepsilon$ and the Fock structure of $(q,2)-$Fock space. This uniqueness is the main content of Theorem \ref{CaCon16}. The construction of ${\mathcal P}_n (\varepsilon)$'s is provided in Theorem \ref{CaCon15}, which states that a typical element of the set ${\mathcal P}_n (\varepsilon)$ comprises

$\bullet$ all pairs $(l^\varepsilon_h, r^\varepsilon_h)$ with the {\bf depth} (see Definition \ref{de-depth} for detail) greater than 1;

$\bullet$ a pair partition of the set $\big\{l^\varepsilon_h, r^\varepsilon_h:\text{the depth of the pair $(l^\varepsilon_h, r^\varepsilon_h)$}$ is equal to either 0 or 1$\big\}$.

Section \ref{DCTsec(q,m)01} is dedicated to doing preparation for presenting and proving our main results, namely, Theorem \ref{CaCon15} and Theorem \ref{CaCon16}. In Section \ref{DCTsec(q,m)01}, we introduce the $(q,2)-$Fock space over a given Hilbert space, as well as the creation and annihilation operators that are defined on this Fock space. Then, we revisit of some properties of pair partitions and non--crossing pair partitions, such as the depth, the restricted crossing number. Finally, we introduce and study the operation known as {\bf gluing of pair partitions} and see some crucial properties of gluing.

\section{The $(q,2)-$Fock space and pair partition } \label{DCTsec(q,m)01}

In this section, we prepare for the statement and proof of our main theorems by undertaking the following tasks:

$\bullet$ introduction of the $(q,2)-$Fock space over a given Hilbert space and an exploration of its elementary properties;

$\bullet$ definition of the creation and annihilation operators, along with an exploration of some of their important properties;

$\bullet$ revisiting some properties related to pair partitions and non--crossing pair partitions, including concepts like the depth of a pair within a non--crossing pair partition, the restricted crossing number of a pair partition;

$\bullet$ introduction of the operation known as {\bf gluing of pair partitions} and an examination of some of its essential properties.

\subsection{Pair partition and non--crossing pair partition of an arbitrary totally ordered set, $(q,2)-$Fock space } \label{DCTsec(q,m)01-1}

Throughout the following, $B^V$ will be used to denote the set of maps from $V$ to $B$ for any sets $B$ and $V$. In particular, for any $m\in\mathbb{N}^*$, when $B=\{-1,1\}$ and $V=\{1,\ldots,m\}$, $B^V$ coincides with $\{-1,1\}^m$ stated in Section \ref{DCT-Intro}, i.e., the set of all $\{-1,1\}-$valued functions defined on $\{1,\ldots,m\}$.

Moreover, for any $n\in\mathbb{N}^*$ and a set  $V:=\{ v_1,\ldots,v_{2n}\}$ ordered as $v_1<\ldots<v_{2n}$,
one introduces $\{-1,1\}^V_\pm$ as a generalization of $\{-1,1\}^{2n}_\pm$ mentioned in Section \ref{DCT-Intro}:
\begin{align*}
&\{-1,1\}^V_+:=\big\{\varepsilon\in \{-1,1\}^V :\sum_{h=1}^{2n}\varepsilon(v_h) =0,\, \sum_{h=p}^{2n}\varepsilon(v_h)\ge 0,\ \forall p\in \{1,\ldots,2n\}\big\}  \notag\\
&\{-1,1\}^V_-:=\{-1,1\}^V\setminus \{-1,1\}^V_+
\end{align*}

\begin{remark}\label{CaCon-rem00} In the definition of $\{-1,1\}^V_+$, the condition $\sum_{h=p}^{2n} \varepsilon(v_h)\ge 0$ for any $p\in \{1,\ldots,2n\}$ can be replaced by $\sum_{h=1}^{k}\varepsilon(v_h)\le 0$ for any $k\in \{1,\ldots,2n\}$ since $\sum_{h=1}^{2n} \varepsilon(v_h)=0$.
\end{remark}

The concept of pair partition can be easily extended as follows: let $V:=\{v_1,\ldots,v_{2n}\}$ be a set of $2n$ elements ordered as $v_1<\ldots<v_{2n}$. Then
$\{(v_{i_h}, v_{j_h})\}_ {h=1}^n$ is a pair partition (or non--crossing pair partition) of $V$ if $\{({i_h}, {j_h})\}_ {h=1}^n$ is a pair partition (or non--crossing pair partition) of $\{1,\ldots,2n\}$. In other words, a pair partition, particularly a non--crossing pair partition, on the set $V:=\{v_1,\ldots,v_{2n}\}$ is defined based on the indices of $v_j$'s.

As generalizations of $PP(2n)$ and $NCPP(2n)$, one denotes
\begin{align*}
PP(V):=&\text{the totality of pair partitions of }V \notag\\
NCPP(V):=&\text{the totality of non--crossing pair partitions of }V
\end{align*}
Then, $PP(2n)=PP(V)\Big\vert_{V=\{1,\ldots,2n\}},\
NCPP(2n)=NCPP(V)\Big\vert_{V=\{1,\ldots,2n\}}$.

{\bf Form now on, for any $n\in{\mathbb N}^*$ and set $V:=\{v_1,\ldots,v_{2n}\}$, the order $v_1<v_2<\ldots <v_{2n}$ will be assumed unless otherwise specified.}

Let $\tau: PP(V)\longmapsto \{-1,1\}^{V}$ be the following map: for any $\theta:=\big\{(v_{l_h},$ $v_{r_h}) \big\}_{h=1} ^n\in PP(V)$, $\varepsilon:=\tau(\theta)$ is defined as
\begin{align*}
\varepsilon(v_j):=\tau(\theta)(v_j):=\begin{cases}
1,&\text{ if }j\in \{r_1,r_2,\ldots, r_n\}\\
-1,&\text{ if }j\in \{l_1,l_2,\ldots, l_n\}\\
\end{cases},\quad \forall j\in\{1,\ldots,n\}
\end{align*}
Clearly, $\tau$ maps $PP(V)$ into $\{-1,1\}^{V}_+$ thanks to the following facts:

$\bullet$ $\sum_{h=1}^{2n}\varepsilon(v_h)= \big\vert\big\{r_h:h \in\{1,\ldots,n\}\big\}\big\vert -\big\vert\big\{l_h:h\in\{1, \ldots,n\}\big\}\big\vert =n-n=0$;

$\bullet$ for any $p\in \{1,2,\ldots,2n\}$, it is clear that $\sum_{h=p}^{2n}\varepsilon(v_h)\ge 0$ because $r_j>l_j$ for any $j\in\{1,\ldots,n\}$.

\noindent On the other hand, for any $n\in\mathbb{N}^*$, $V:=\{v_1,\ldots,v_{2n}\}$ 
and $\varepsilon\in \{-1,1\}^{V}_+$, we can generalize the sets $PP(2n, \varepsilon)$ and $NCPP(2n,\varepsilon)$ as follows:
\begin{align*}
&PP(V,\varepsilon):=\tau^{-1}(\varepsilon)\notag\\
:=&\big\{\{(v_{l_h},v_{r_h})\}_{h=1}^n\in PP(V): \varepsilon(v_{l_h})=-1,\, \varepsilon(v_{r_h})=1, \forall h=1,\ldots,n\big\}\notag\\
&NCPP(V,\varepsilon):=\tau^{-1}(\varepsilon)\cap NCPP(V)
\end{align*}
Moreover, by employing the map $\tau$ described above and utilizing the results concerning the cardinalities of $PP(2n,\varepsilon)$ and $NCPP(2n,\varepsilon)$ as discussed in Section \ref{DCT-Intro}, one has
\[
\big\vert\tau^{-1}(\varepsilon)\big\vert=\prod_{h=1}^n (2h-l_h);\qquad
\big\vert\tau^{-1}(\varepsilon)\cap NCPP(V)\big\vert=1
\]
where, the second equality shows that the application $\tau$ induces a {\bf bijection} between $NCPP(V)$ and $\{-1,1\}^{V}_+$.

In the following, for any $n\in\mathbb{N}^*$ and the set $V:=\{v_1,\ldots,v_{2n}\}$, we will refer to the following:

$\bullet$ $\tau(\theta)\in \{-1,1\}^{V}_+$ as the {\bf counterpart} of $\theta$ for any $\theta\in NCPP(V)$;

$\bullet$ the unique element of  $\tau^{-1} (\varepsilon)\cap NCPP(V)$ as the {\bf counterpart} of $\varepsilon$ for any $\varepsilon\in \{-1,1\}^{V}_+$.

The counterpart of $\varepsilon\in\{-1,1\}^{V}_+$ mentioned above will be denoted as $\{(v_{l^\varepsilon _h}, v_{r^\varepsilon_h})\}_{h=1}^n$ which is identical to $\{(l^\varepsilon_h, r^\varepsilon_h)\}_ {h=1}^n$ when $v_k=k$ for all $k\in\{1,\ldots,2n\}$.

\begin{remark}\label{CaCon-rem01} Notice that for any
$n\in\mathbb{N}^*$, $V:=\{v_1,\ldots,v_{2n}\}$ and $\varepsilon\in \{-1,1\}^{V}_+$, all elements in $PP(V,\varepsilon)$ must have the same left--indices (respectively, right--indices), i.e., the $n$ elements of the set $\varepsilon^{-1}(\{-1\})$ (respectively, $\varepsilon^{-1}(\{1\})$).
\end{remark}

Now we begin to explore the $(q,2)-$Fock space, which is a particular $(q,m)-$Fock space introduced in \cite{YGLu2022a}. Let $\mathcal{H}$ be a Hilbert space with a scalar product $\langle \cdot,\cdot \rangle$  of the dimensional greater than or equal to 2 (this convention will be maintained throughout), let $\mathcal{H}^{\otimes n}$ be its $n-$fold tensor product for any $n\ge 2$. One defines, for any $q\in[-1,1]$,

$\bullet$ $\lambda_1:={\bf 1}_{\mathcal{H}}$, i.e., the identity operator on $\mathcal{H}$;

$\bullet$ for any $n\in\mathbb{N}$, $\lambda_{n+2}$ be the {\bf linear} operator on $\mathcal{H}^{\otimes (n+2)}$ such that
\begin{equation*}
\lambda_{n+2}:={\bf 1}_{\mathcal{H}}^{\otimes n}\otimes \lambda_2\,\text{ and } \lambda_2(f\otimes g):=f\otimes g+q g\otimes f\,,\quad\forall f,g\in\mathcal{H}
\end{equation*}

The positivity of $\lambda_n$'s can be easily verified (see, for instance, \cite{BoKumSpe97}, \cite{Bo-Spe91}, \cite{Ji-Kim2006}) and which ensures that
\[
\mathcal{H}_n:=\text{ the completion of the } \big(\mathcal{H}^{\otimes n}, \langle \cdot,\lambda_{n} \cdot \rangle_{\otimes n}\big)/Ker\langle \cdot, \lambda_{n}\cdot \rangle_{\otimes n}
\]
is a Hilbert space, where $\langle \cdot,\cdot \rangle_{\otimes n}$ is the usual tensor scalar product. One denotes the scalar product of $\mathcal{H}_n$ as $\langle \cdot,\cdot \rangle_{n}$, where $\langle \cdot,\cdot \rangle_{1}:=\langle \cdot,\cdot \rangle $ and for any $n\ge2$,
\begin{equation*}
\langle F,G \rangle_{n}:=\langle F,\lambda_nG \rangle_{\otimes n} \,,\quad\forall F,G\in \mathcal{H}^{\otimes n}
\end{equation*}
or equivalently for any $n\in \mathbb{N}$, for any
$F\in \mathcal{H}^{\otimes (n+2)}$, $G\in \mathcal{H}^{\otimes n}$, and for any $f,g\in \mathcal{H}$,
\begin{align*}
\langle F,G\otimes f\otimes g \rangle_{n+2}:=&\langle F,G\otimes f\otimes g\rangle_{\otimes (n+2)}+ q\langle F,G\otimes g\otimes f\rangle_{\otimes (n+2)}
\end{align*}

\begin{definition}\label{(q,2)-Fock}Let $\mathcal{H}$ be a Hilbert space, and for any $q\in[-1,1]$ and $n\in \mathbb{N}^*$, let $\mathcal{H}_n$ be the Hilbert space as described above. The Hilbert space $\Gamma_{q,2} (\mathcal{H}) :=\bigoplus_{n=0}\mathcal{H}_n $ is referred to as the {\bf $(q,2)-$Fock space} over $\mathcal{H}$, where $\mathcal{H}_{0} :=\mathbb{C}$. Moreover,

$\bullet$ the vector $\Phi:=1\oplus0\oplus0 \oplus\ldots$ is called the {\bf vacuum vector} of $\Gamma_{q,2}(\mathcal{H})$;

$\bullet$ for any $n\in\mathbb{N}^*$, $\mathcal{H}_n$ is named as the {\bf $n-$particle space}.
\end{definition}

Throughout, we will use the symbol $\langle \cdot,\cdot\rangle$ and $\Vert\cdot\Vert$ to denote the scalar product and the induced norm on both $\Gamma_{q,2}(\mathcal{H})$ and $\mathcal{H}_n$'s.

It is straightforward to observe the following {\bf consistency} of $\langle \cdot,\cdot \rangle_{n}$'s: for any non--zero $f\in\mathcal{H}$ and any $n\in\mathbb{N}^*$, holds the following:
\begin{equation*}
\Vert f\otimes F\Vert=0\text{ whenever } F\in\mathcal{H}_{n} \text{ has norm zero}
\end{equation*}
This consistency guarantees that, for any $f\in\mathcal {H}$, the operator that maps $F\in\mathcal{H}_{n}$ to
$f\otimes F\in \mathcal{H}_{n+1}$ is a well--defined linear operator from $\mathcal{H}_{n}$ to $\mathcal{H}_{n+1}$.

\begin{definition}\label{creaOn(q,2)} For any $f\in \mathcal{H}$, the {\bf $(q,2)-$creation operator} (with the test function $f$) $A^+(f)$ is defined as such a {\bf linear} operator on $\Gamma_{q,2}(\mathcal{H})$ that
\begin{equation*}
A^+(f)\Phi:=f\,,\quad A^+(f)F:=f\otimes F,\quad \forall n\in\mathbb{N}^*\text{ and }F\in\mathcal{H}_n
\end{equation*}
\end{definition}

The following result was proved in \cite{YGLu2022a}.
\begin{proposition}\label{DCT05g}Let $\mathcal{H}$ be a Hilbert space and let $q$ belong to $[-1,1]$. For any $f\in\mathcal{H}$, the $(q,2)-$creation operator $A^+(f)$ is bounded and 
\begin{align*}
\Vert A^+(f)\Vert =\Vert f\Vert\cdot\begin{cases} \sqrt{1+q},&\text{ if }q\in[0,1];\\ 1,&\text{ if }q\in[-1,0)  \end{cases}
\end{align*}
So, $A(f):=\big(A^+(f)\big)^*$ is well--defined, named as the {\bf $(q,2)-$annihilation operator} with the test function $f\in \mathcal{H}$. Moreover, for any $f\in \mathcal{H}$, hold the following statements:

1) $A(f)\Phi=0$ and for any $n\in\mathbb{N}^*$, $\{g_1,\ldots, g_n\}\subset\mathcal{H}$,
\begin{align}\label{DCT05g1}
&A(f)(g_1\otimes\ldots\otimes g_n)=A(f)A^+(g_1) \ldots A^+(g_1)\Phi\notag\\
=&\begin{cases}\langle f,g_1\rangle\Phi,&\text{ if } n=1;\\  \langle f, g_1\rangle g_2+q\langle f, g_2\rangle g_1,&\text{ if }n=2;\\  \langle f, g_1\rangle g_2\otimes\ldots\otimes g_n,&\text{ if }n> 2  \end{cases}
\end{align}

2) $\Vert A(f)\big\vert_{\mathcal{H}_{1}}\Vert =\Vert A^+(f)\big\vert_{\mathcal{H}_{0}}\Vert =\Vert f\Vert$ and for any $n\in\mathbb{N}^*$,
\begin{align*}
&\Vert A(f)\big\vert_{\mathcal{H}_{n+1}}\Vert =\Vert A^+(f)\big\vert_{\mathcal{H}_n}\Vert\notag\\ &
\Vert A(f)A^+(f)\big\vert_{\mathcal{H}_n}\Vert  =\Vert A^+(f)A(f)\big\vert_{\mathcal{H}_n}\Vert =\Vert A(f)\big\vert_{\mathcal{H}_n}\Vert ^2
\end{align*}
hereinafter, for any $m\in{\mathbb N}$ and for any linear operator $B$ on the $(q,2)-$Fock space, $B\big\vert_{{\mathcal H}_m}$ is defined as its restriction to ${\mathcal H}_m$. Consequently
\begin{align*}
\Vert A(f)\Vert =\Vert A^+(f)\Vert\,;\quad \Vert A(f)A^+(f)\Vert  =\Vert A^+(f)A(f)\Vert =\Vert A(f)\Vert^2
\end{align*}
\end{proposition}

Furthermore, we have the following result:
\begin{proposition}\label{DCT28}Let $\mathcal{H}$ be a Hilbert space and $\{f_k\}_{k=1}^\infty\subset \mathcal{H}$. For any $q\in[-1,1]$, on the $(q,2)-$Fock space over $\mathcal{H}$,

1) the vacuum expectation
\begin{align}\label{DCT28a}
\big\langle \Phi,A^{\varepsilon(1)}(f_1)\ldots A^{\varepsilon(m)}(f_m)\Phi\big\rangle
\end{align}
differs from zero only if $m$ is even, i.e. $m=2n$, and
$\varepsilon$ belongs to $\{-1,1\}^{2n}_+$;

2) for any $n\in{\mathbb N}^*$ and $\varepsilon\in  \{-1,1\}^{2n}_+$, and for any $C\in {\mathbb C}$, the equality $A^{\varepsilon(1)}(f_1)\ldots A^{\varepsilon(2n)}(f_{2n})\Phi=C\Phi$ holds if and only if $\big\langle \Phi,A^{\varepsilon(1)}(f_1)\ldots$ $ A^{\varepsilon(2n)}(f_{2n})\Phi\big\rangle=C$.
\end{proposition}
\begin{proof}It is evident that, for any $m\in{\mathbb N}^*$ and $\varepsilon\in  \{-1,1\}^{m}$, the vector 
$A^{\varepsilon(1)}(f_1)\ldots A^{\varepsilon(m)}(f_m) \Phi$ differs from zero only when the condition
\begin{align}\label{DCT28c1}
\sum_{k=p}^m \varepsilon(k)\ge0,\quad\forall p\in \{1,\ldots, m\}
\end{align}
is satisfied. This is due to the following observations:
\begin{align*}
\mathcal{H}_n\overset{A^+(f)}{\longmapsto} \mathcal{H}_{n+1}\overset{A(g)}{\longmapsto}\mathcal{H}_n, \quad A(g)\Phi=0,\quad\forall n\in\mathbb{N},\ f,g\in\mathcal{H}
\end{align*}
Which ensure that if \eqref{DCT28c1} is not met, sooner or later, an annihilation operator acts on the vacuum vector, resulting in zero. Similarly, the vector $\big(A^{\varepsilon(1)}(f_1)\ldots A^{\varepsilon(m)}(f_m)\big)^*\Phi$, i.e. the vector $A^{-\varepsilon(m)}(f_m)\ldots A^{-\varepsilon(1)}(f_1)\Phi$ differs from zero only if
\begin{align}\label{DCT28c2}
\sum_{k=1}^p\varepsilon(k)\le0,\quad\forall p\in\{1,\ldots,m\}
\end{align}
In particular, the vacuum expectation \eqref{DCT28a} differs from zero only if, by taking $p=1$ in \eqref{DCT28c1} and $p=m$ in \eqref{DCT28c2},
\begin{align}\label{DCT28c3}
\sum_{k=1}^m \varepsilon(k)=0
\end{align}
this condition clearly necessitates that $m$ must be even.

In the case of $m=2n$, \eqref{DCT28c3} and \eqref{DCT28c1} become to respectively $\sum_{k=1}^{2n} \varepsilon(k)=0$ and $\sum_{k=p}^{2n} \varepsilon(k)\ge0$ for all $p\in \{1,\ldots, m\}$. That is the expectation \eqref{DCT28a} (now referred to as \eqref{DCT02c09}) differs from zero only if $\varepsilon\in  \{-1,1\}^{2n}_+$.

The first equality in the affirmation 2) gives trivially the second. On the other hand, one notices that the vector $A^{\varepsilon (1)}(f_1) \ldots A^{\varepsilon(2n)}(f_{2n})\Phi$ must be of the form $c\Phi$ whenever $\varepsilon\in \{-1,1\}^{2n}_+$. So, the second equality in the affirmation 2) 
ensures that $C=c$.      \end{proof}

\subsection{Restricted crossing number and depth} \label{DCTsec(q,m)01-2}

In this subsection, we will revisit two fundamental concepts:

$\bullet$ the {\bf restricted crossing number} of a pair partition;

$\bullet$ the {\bf depth} of a pair in a given non--crossing pair partition.

\begin{definition}\label{def-cro} For any $n\in\mathbb{N}^*$ and a set $V:=\{v_1,\ldots,v_{2n}\}$,
for any $\theta=\{(v_{l_h}, v_{r_h})\} _{h=1}^{n}\in PP(V)$, we define the {\bf restricted crossing number of $\theta$} as follows (adapting the convention $\sum_{h=1} ^{0}:=0$):
\begin{align*}
c(\theta):=c(\{(v_{l_h},v_{r_h})\}_{h=1}^{n}):=\sum_{h=1} ^{n-1}\vert\{j: l_h<l_j<r_h<r_j \}\vert
\end{align*}
\end{definition}

The concept of restricted crossing number is widely used in various contexts, particularly in the derivation of Wick's theorem in different cases (see, e.g., \cite{Biane97}, \cite{BoKumSpe97}, \cite{Bo-Spe91}, \cite{EffPopa}, \cite{jK-mK1992}, \cite{YGLu2022a}, \cite{tM-mS2008}, \cite{tM-mS-sS2007}).

\begin{remark}\label{CaCon16b0} 1) It is obvious that
$c(\theta)=0$ if and only if $\theta$ is non--crossing.

2) Clearly, $c(\theta)$ depends solely on the indices of $v_h$'s but not their specific nature. So it is common to write $c(\{(v_{l_h},v_{r_h})\}_{h=1}^n)$ simply as $c(\{(l_h,r_h)\}_{h=1}^{n})$.
\end{remark}

\begin{proposition}\label{CaCon17} Let $n\in{\mathbb N}^*$, $V:=\{v_1,\ldots,v_{2n}\}$ 
and $\theta:=\{(v_{l_h},v_{r_h})\} _{h=1}^n\in PP(V)$. We have the following assertions:

1) For any totally ordered set $U$ with a cardinality greater than or equal to $2n$ and any order--preserving map $\rho:V\mapsto U$, by denoting $\rho(V):=\{\rho(v):\,v\in V\}$ and $\rho(\theta):= \{(\rho(v_{l_h}), \rho(v_{r_h})\} _{h=1}^n$, the following equalities hold:
\begin{align*}
PP(\rho(V))=\big\{\rho(\theta): \theta\in PP(V) \big\}\,;\quad c(\rho(\theta))=c(\theta)\,,\quad\forall \theta\in PP(V)
\end{align*}

2) In case $n\ge2$, for any $k\in\{1,\ldots,n\}$, it must be true that $\theta_k:=\{(v_{l_h},v_{r_h})\} _{1\le h\ne k\le n}$ belongs to $PP(V\setminus\{v_{l_k},v_{r_k}\})$, and  holds the following relationship:
\begin{align}\label{CaCon17b}
c(\theta)= c(\theta_k)+\big\vert\{h:l_h<l_k<r_h<r_k\} \big\vert +\big\vert\{h:l_k<l_h<r_k<r_h\} \big\vert
\end{align}
In particular,
\begin{align}\label{CaCon17c}
c(\theta)=c(\theta_k)\ \ \text{ if } r_k=l_k+1;\qquad
c(\theta)=c(\theta_k)+1\ \ \text{ if } r_k=l_k+2
\end{align}
and
\begin{align}\label{CaCon17d}
c(\theta)=c(\theta_n)+r_n-l_n-1
\end{align}
\end{proposition}
\begin{proof} Let's firstly prove \eqref{CaCon17b}. According to the definition, we have
\begin{align}\label{CaCon17e}
c(\theta_k):=\sum_{h\in\{1,\ldots,n-1\}\setminus\{k\}} \big\vert\{j\ne k: l_h<l_j<r_h<r_j \}\big\vert
\end{align}
and
\begin{align}\label{CaCon17f}
c(\theta):=&\sum_{h\in\{1,\ldots,n-1\}} \big\vert\{j: l_h<l_j<r_h<r_j \}\big\vert\notag\\
=&\sum_{h\in\{1,\ldots,n-1\}\setminus\{k\}}\big\vert\{j:l_h<l_j<r_h <r_j \}\vert+\vert\{j: l_k<l_j<r_k<r_j \}\big\vert
\end{align}
Moreover, for any $k\in\{1,\ldots,n-1\}$, the first term on the right hand side of \eqref{CaCon17f} equals to
\begin{align*}
\sum_{h\in\{1,\ldots,n-1\}\setminus\{k\}}
\big\vert\{j\ne k: l_h< l_j<r_h <r_j\}\big\vert+
\big\vert\{h: l_h<l_k<r_h<r_k \}\big\vert
\end{align*}
which is nothing else than $c(\theta_k)+\big\vert\{h: l_h<l_k<r_h<r_k \}\big\vert$ due to \eqref{CaCon17e}. By applying this result to \eqref{CaCon17f}, we have successfully demonstrated the validity of \eqref{CaCon17b}.

The statement 1) follows directly from the definition. Now, let's proceed with the proof of the statement 2).

It is evident that $\theta_k:=\{(v_{l_h},v_{r_h})\} _{1\le h\ne k\le n}\in PP(V\setminus\{v_{l_k}, v_{r_k}\})$. The formulae \eqref{CaCon17c} and \eqref{CaCon17d} are established just as corollaries of \eqref{CaCon17b} and the following facts:

$\bullet$ $\{h:l_h<l_k<r_h<r_k\}= \emptyset=\{h:l_k<l_h<r_k<r_h\}$ when $r_k=l_k+1$.

$\bullet$ If $r_k=l_k+2$ (equivalently, $l_k+1=r_k-1$), we can confirm that $l_k+1$ could be either $l_{k+1}$ or $r_p$ for some $p<k$. In other words, either $\{h:l_h<l_k<r_h<r_k\}=\{p\}$ or $\{h:l_k<l_h<r_k<r_h\}=\{l_{k+1}\}$. Thus $\big\vert\{h:l_h<l_k<r_h<r_k\} \big\vert +\big\vert\{h:l_k<l_h<r_k<r_h\} \big\vert=1$.

$\bullet$ The fact $\{l_n+1,\ldots,2n\}\subset \{r_h:h=1,\ldots,n\}$ guarantees that
\[\{h:l_n<l_h<r_n<r_h\}=\emptyset;\quad \{h:l_h<l_n<r_h<r_n\}= \{l_n+1,\ldots,r_n-1\}\]
Consequently, $\big\vert\{h:l_n<l_h<r_n<r_h\} \big\vert=0$ and $\big\vert\{h:l_h<l_n<r_h<r_n\} \big\vert =r_n-l_n-1$.         \end{proof}

\begin{definition}\label{de-depth} For any $n\in\mathbb {N}^*$ and $V:=\{v_1,\ldots,v_{2n}\}$, 
for any $\theta:=\{(v_{l_h}, v_{r_h})\} _{h=1}^n \in NCPP(V)$ and $k\in \{1,\dots ,n\}$, the expression
\begin{align}\label{CaCon17k}
d_\theta(v_{l_k},v_{r_k}):=\big\vert \left\{h\in \{1,\dots ,n\} \ : \ l_{h}< l_{k}<r_{k}< r_{h}\right\}\big\vert
\end{align}
will be named as the {\bf depth} of the pair $(v_{l_k},v_{r_k})$  in $\theta$. Moreover, for any $\varepsilon\in\{-1,1\}^V _+$ with the counterpart of $\theta$, we denote $d_{\varepsilon}:=d_{\theta}$.
\end{definition}

It is important to note that we have employed the concept of {\it the depth of a pair} in a given pair partition $\theta$ only under the condition that $\theta$ is non--crossing, even though it could be defined in a more general case.

\begin{remark}\label{CaCon17g} Let $n\in\mathbb{N}^*$ and $V:=\{v_1,\ldots,v_{2n}\}$,
let $\theta:=\{(v_{l_h},v_{r_h})\} _{h=1}^n \in NCPP(V)$.

1) In analogy to the fact mentioned in Definition \ref{CaCon16b0}, for any $k\in \{1,\dots ,n\}$,  the depth of the pair $(v_{l_k},v_{r_k})$ in a given non--crossing pair partition $\theta$ depends only on the indices of the elements (regardless of their specific values) in $V$. For this reason,  it is common to denote the depth as $d_\theta(l_k,r_k)$ instead of $d_\theta(v_{l_k},v_{r_k})$, treating $V$ as if it were $\{1,\ldots,2n\}$.

2) For any $ k\in \{1,\dots ,n\}$, the depth of the pair $(v_{l_k},v_{r_k})$ in the non--crossing pair partition $\theta$ is {\bf the number of pairs in $\theta$ that include the pair $(v_{l_k},v_{r_k})$}. In particular,

$\bullet$ by recalling that $l_1$ must be $1$ and $2n$ must be equal to some $r_m$ with $m\in\{1,\ldots,n\}$, we have
\begin{align}\label{CaCon17h}
d_{\theta}(v_{l_1},v_{r_1})=0=d_{\theta}(v_{l_m},v_{2n})
\end{align}

$\bullet$ for any $k\in \{2,\dots ,n\}$, \eqref{CaCon19a1} together with the increasing property of left--indices indicate that
\begin{align*}
d_{\theta}(v_{l_k},v_{r_k})=0 &\iff \text{there is {\bf no} $1\le s< n$ such that }l_{s}<l_k<r_k<r_{s} \notag\\
&\iff \text{there is {\bf no} $1\le s<k$ such that }l_{s}<l_k<r_k<r_{s} \notag\\
&\iff \text{there is {\bf no} $1\le s<k$ such that } r_k<r_{s}
\end{align*}
\end{remark}

\begin{proposition}\label{CaCon18} For any $n\in\mathbb{N}^*$ and $V:=\{v_1,\ldots,v_{2n}\}$, 
for any $\theta:=\{(v_{l_h},v_{r_h})\} _{h=1}^n \in NCPP(V)$ and  $k\in \{1,\dots ,n\}$,
\begin{equation}\label{CaCon18a}
d_{\theta}(v_{l_k},v_{r_k})=\big\vert\big\{h:r_h>r_k\big\} \big\vert-\big\vert\big\{h:l_h>r_k\big\} \big\vert
\end{equation}
In particular,
\begin{equation}\label{CaCon18b}
d_{\theta}(v_{l_n},v_{r_n})=2n-r_n=2n-l_n-1
\end{equation}
\end{proposition}
\begin{proof} Without loss of generality, we assume that $V=\{1,\ldots,2n\}$.

The second equality in \eqref{CaCon18b} is trivial since the non--crossing property guarantees $r_n=l_n+1$. Additionally, any element in the set $\{r_n+1,\ldots,2n\}$ must be a certain $r_h$ with $h\in\{1,\ldots,n-1\}$, which implies that $\big\vert\big\{h:r_h>r_n\big\} \big\vert=2n-r_n$ and $\big \vert\big\{h:l_h>r_n\big\} \big\vert=0$. Therefore, the first equality in \eqref{CaCon18b} is a direct consequence of \eqref{CaCon18a}.

Let $R^{(l)}_k:=\{h:l_h>r_k\}$ and $R^{(r)}_k:=\{h:r_h>r_k\}$, the proof of \eqref{CaCon18a} can be completed as follows:

$\bullet$ the fact of $r_p>l_p$ for any $p$ makes sure that $R^{(l)}_k\subset R^{(r)}_k$ and as a result,
the right hand side of \eqref{CaCon18a} equals to $\big\vert R^{(r)}_k\setminus R^{(l)}_k\big\vert$;

$\bullet$ $p\in R^{(r)}_k\setminus R^{(l)}_k$ if and only if $l_p<r_k<r_p$ if and only if  $l_p<l_k<r_k<r_p$; in other words, $R^{(r)}_k\setminus R^{(l)}_k$ is indeed the set on the right hand side of \eqref{CaCon17k}.
\end{proof}

\begin{proposition}\label{CaCon25}
For any $n\in\mathbb{N}^*$ and $V:=\{v_1,\ldots, v_{2n}\}$, 
for any $\theta:=\{(v_{l_h},v_{r_h})\} _{h=1}^n \in NCPP(V)$, the following statements hold:

1) For any $k\in \{1,\dots ,n\}$, the pair partition $\theta_k:= \{(v_{l_h}, v_{r_h})\}_{1\le h\ne k\le n}$ belongs to $NCPP(V\setminus \{v_{l_k},v_{r_k}\})$. Furthermore, for any $k\ne p\in\{1,\dots ,n\}$,
\begin{align}\label{CaCon25b}
d_{\theta_k}(v_{l_p},v_{r_p})&=\big\vert\big\{h\ne k:r_h>r_p\big\} \big\vert-\big\vert\big\{h\ne k:l_h>r_p\big\} \big\vert\notag\\
&=\begin{cases}d_{\theta}(v_{l_p},v_{r_p})-1,&\text{ if } l_k<l_p<r_p<r_k\\ d_{\theta}(v_{l_p},v_{r_p}), &\text{ otherwise} \end{cases}\notag\\
&\ge \begin{cases}d_{\theta}(v_{l_k},v_{r_k}),&\text{ if } l_k<l_p<r_p<r_k\\    d_{\theta}(v_{l_p},v_{r_p}), &\text{ otherwise}\end{cases}
\end{align}

2) If $l_n=2n-2$ (so $r_n=2n-1$) and $2n=r_m$ (it must be true that $m\le n-1$ since $r_n=2n-1$), for any $r\in\{2n-1,2n\}$, by denoting  $r'_m:=r$ and $r'_h:=r_h$ for any $h\in\{1,\ldots,n-1\} \setminus\{m\}$, $\theta':=\{(v_{l_h},v_{r'_h})\} _{h=1}^{n-1}$ must be an element of $NCPP(\{v_1,\ldots, v_{2n-3},v_r\})$ and
\begin{equation*}
d_{\theta}(v_{l_p},v_{r_p})=d_{\theta'}(v_{l_p},v_{r'_p}),\quad\forall p\in\{1,\ldots,n-1\}
\end{equation*}
more particularly, $d_{\theta}(v_{l_m},v_{r_m})=0= d_{\theta'}(v_{l_m},v_{r'_m})$.

3) For any different $k,p\in\{1,\dots ,n\}$,
\begin{align}\label{CaCon25c}
d_{\theta_k}(v_{l_p},v_{r_p})
=d_{\theta}(v_{l_p},v_{r_p})
\end{align}
if either the following holds:

$\bullet$ $r_k=l_k+1$;

$\bullet$ $d_\theta(v_{l_k},v_{r_k})\ge s$ and $d_\theta (v_{l_p},v_{r_p})\le s$ for some $s\in\{0,1,\ldots,n-1\}$.

$\bullet$ $d_\theta(v_{l_k},v_{r_k})\ge s$ and $d_{\theta _k}(v_{l_p},v_{r_p})< s$ for some $s\in\{1,\ldots,n-1\}$.
\end{proposition}
\begin{proof}The non--crossing property ensures that $\theta_k:=\{(v_{l_h},v_{r_h})\}_{1\le h\ne k\le n}$ belongs to $NCPP(V\setminus \{v_{l_k},v_{r_k}\})$ when $\theta=\{(v_{l_h},v_{r_h})\}_{h=1}^n\in NCPP(V)$.

For any $p\ne k$, \eqref{CaCon18a} clearly leads to the first equality in \eqref{CaCon25b}. Moreover, the non--crossing property dictates that one of the following cases must occur:
\[r_p<l_k\,;\qquad  r_k<l_p\,;\qquad  l_p<l_k<r_k<r_p\,;\qquad l_k<l_p<r_p<r_k\]
The second equality in \eqref{CaCon25b} will be proven for each of these cases.

$\bullet$ When $r_p<l_k$ (or equivalently, $l_p<r_p<l_k <r_k$), we observe that $k$ belongs to both $\big\{h:r_h>r_p\big\} $ and $\big\{h:l_h>r_p\big\} $. So
\begin{align*}\big\vert \big\{h:r_h>r_p\big\} \big\vert=&\big\vert\big\{h\ne k:r_h>r_p\big\} \big\vert+1\\
\big\vert \big\{h:l_h>r_p\big\} \big\vert=&\big\vert\big\{h\ne k:l_h>r_p\big\} \big\vert+1
\end{align*}
Consequently, \eqref{CaCon18a} guarantees that
\begin{align*}
d_\theta(l_p,r_p)=& \big\vert \big\{h:r_h>r_p\big\} \big\vert-\big\vert\big\{h:l_h>r_p\big\} \big\vert\\
=&\big\vert\big\{h\ne k:r_h>r_p\big\} \big\vert+1 -\big(\big\vert\big\{h\ne k:l_h>r_p\big\} \big\vert+1\big)=d_{\theta_k}(l_p,r_p)
\end{align*}

$\bullet$ If either $r_k<l_p$ (equivalently, $l_k<r_k<l_p<r_p$) or $l_p<l_k<r_k<r_p$,
$k$ belongs to neither $\big\{h:r_h>r_p\big\} $ nor $\big\{h:l_h>r_p\big\} $. In these cases,
\begin{align*}\big\vert \big\{h:r_h>r_p\big\} \big\vert=\big\vert\big\{h\ne k:r_h>r_p\big\} \big\vert;\ \  \big\vert \big\{h:l_h>r_p\big\} \big\vert=\big\vert\big\{h\ne k:l_h>r_p\big\} \big\vert
\end{align*}
So, \eqref{CaCon18a} implies that
\begin{align*}
d_\theta(l_p,r_p){=}& \big\vert \big\{h:r_h>r_p\big\} \big\vert-\big\vert\big\{h:l_h>r_p\big\} \big\vert\\
=&\big\vert\big\{h\ne k:r_h>r_p\big\} \big\vert -\big\vert\big\{h\ne k:l_h>r_p\big\} \big\vert
=d_{\theta_k}(l_p,r_p)
\end{align*}

$\bullet$ If $l_k<l_p<r_p<r_k$, $k$ belongs to $\big\{h:r_h>r_p\big\}$ but not to $\big\{h:l_h>r_p\big\}$. So,
\begin{align*}\big\vert \big\{h:r_h>r_p\big\} \big\vert=&\big\vert\big\{h\ne k:r_h>r_p\big\} \big\vert+1\\ \big\vert \big\{h:l_h>r_p\big\} \big\vert=&\big\vert\big\{h\ne k:l_h>r_p\big\} \big\vert
\end{align*}
Consequently, \eqref{CaCon18a} ensures that
\begin{align*}
&d_\theta(l_{p},r_{p})=
\big\vert \big\{h:r_h>r_p\big\} \big\vert-\big\vert\big\{h:l_h>r_p\big\} \big\vert\\
=&\big\vert\big\{h\ne k:r_h>r_p\big\} \big\vert+1 -\big\vert \big\{h\ne k:l_h>r_p\big\}\big\vert= d_{\theta_k}(l_{p},r_{p})+1
\end{align*}

As a result, we obtain the inequality in \eqref{CaCon25b} as follows: the fact $l_k<l_p<r_p<r_k$ surely implies $d_{\theta}(l_{p},r_{p})\ge d_{\theta} (l_{k},r_{k})+1$ and so $d_{\theta_k}(l_{p},r_{p})= d_{\theta}(l_{p},r_{p})-1\ge d_{\theta}(l_{k},r_{k})$.
Thus the affirmation 1) is proven. Now we are going to prove the affirmation 2).

It is obvious that $\theta':=\{(v_{l_h}, v_{r'_h}) \}_{h=1}^{n-1}\in NCPP(\{v_1,\ldots, v_{2n-3},v_r\})$ if $l_n=2n-2$ and $r\in\{2n-1,2n\}$. Moreover

$\bullet$ \eqref{CaCon17h} implies that $d_{\theta}(l_m,r_m)=0$  because $r_m=2n$; on the other hand, \eqref{CaCon17h} gives $d_{\theta'}(l_m,r'_m)=0$ since $r'_m:=r$ is the rightmost element of the set $\{1,\ldots, 2n-3,r\}$;

$\bullet$ for any $k\in\{1,\ldots,n-1\}\setminus\{m\}$, it can be observed that
$r'_k=r_k\le 2n-3<l_n=2n-2<r_n$, therefore
\begin{align*}
&d_{\theta}(l_k,r_k)\overset {\eqref{CaCon18a}}{=}\big\vert\big\{h\le n:r_h>r_k\big\} \big\vert - \big\vert\big\{h\le n:l_h>r_k\big\} \big\vert\\
=&\big\vert\big\{h\le n-1:r_h>r_k\big\}\cup\{n\} \big\vert -\big(\big\vert\big\{h\le n:l_h>r_k\big\} \cup\{n\}\big\vert\big)\\
=&\big\vert\big\{h\le n-1:r'_h>r'_k\big\} \big\vert+1 -\big\vert\big\{h\le n-1:l_h>r'_k\big\} \big\vert-1\overset{\eqref{CaCon18a}}{=} d_{\theta'}(l_k,r'_k)
\end{align*}

Finally, we prove the affirmation 3). If we are in one of three cases mentioned in the affirmation 3), the condition $l_k<l_p <r_p<r_k$ is impossible:

$\bullet$ The condition $r_k=l_k+1$ guarantees clearly $\{p:l_k<l_p <r_p<r_k\}=\emptyset$.

$\bullet$ If $d_\theta(l_k,r_k)\ge s$ and $l_k<l_p<r_p< r_k$, one gets $d_\theta(l_p,r_p)\ge s+1$, which is contradictory to the assumption $d_\theta(l_p,r_p)\le s$.

$\bullet$ If $d_\theta(l_k,r_k)\ge s$ and $l_k<l_p <r_p<r_k $, it follows that $d_{\theta}(l_p,r_p)\ge s+1$, and so $d_{\theta _k}(l_p,r_p)\ge s$. This contradicts the assumption $d_{\theta_k}(l_p,r_p)< s$.\\
As a consequence, the second equality in \eqref{CaCon25b} yields \eqref{CaCon25c}. \end{proof}

\begin{corollary} Using the assumption and notations of Proposition \ref{CaCon25},

$\bullet$ in case $d_\theta(v_{l_k},v_{r_k})\ge 1$,
\begin{align}\label{CaCon25d}
d_{\theta_k}(v_{l_p},v_{r_p})=0\text{ if and only if }d_{\theta}(v_{l_p},v_{r_p})=0
\end{align}

$\bullet$ in case $d_\theta(v_{l_k},v_{r_k})\ge2$,
\begin{align}\label{CaCon25e} d_{\theta_k}(v_{l_p},v_{r_p})=1  \text{ if and only if } d_\theta(v_{l_p},v_{r_p})=1 \end{align}
\end{corollary}
\begin{proof} The proof is completed just by taking $s=1$ and $2$ respectively in the second case mentioned in the affirmation 3) of Proposition \ref{CaCon25}.    \end{proof}

\subsection{Gluing of pair partitions} \label{DCTsec(q,m)01-3}

Our main goal in this subsection is to explore the process of {\bf gluing} pair partitions. Let

$\bullet$ $m\ge 2$ and $\{n_1,\ldots,n_m\} \subset\mathbb{N}^*$;

$\bullet$ $V_p:=\{v_{p,1},\ldots,v_{p,2n_p}\}$ be a set
with the order $v_{p,1}<\ldots<v_{p,2n_p}$, where, $p\in\{1,\ldots,m\}$; moreover, $V_p$'s are assumed pairwise disjoint;

$\bullet$ ``$<$'' be a total order on $V:=\bigcup_p V_p$ such that its restriction to each $V_p$ follows the order $v_{p,1}<\ldots<v_{p,2n_p}$.

It is essential to remark that, under the order ``$<$'' on $V$, there are no specific rule for combining   $v'\in V_p$ and $v''\in V_q$ when $p\ne q$.

By denoting $n:=n_1+\ldots+n_m$, the disjointness of $V_p$'s implies that the cardinality of $V$ is equal to the sum of the cardinalities of $V_p$ for $p$ ranging from 1 to $m$, i.e., $\vert V\vert=\sum_{p=1}^m\vert V_p\vert=2\big(n_1+\ldots+n_m\big)=2n$. For any pair partitions $\theta^{(p)}:=\big\{(v_{p,l^{(p)}_h} ,v_{p, r^{(p)}_h})\big\}_{h=1}^{n_p}\in PP(V_p)$ with $p\in\{1, \ldots,m\}$, one can {\it glue} $\theta^{(p)}$'s together by introducing:
\begin{align*}
&\theta^{(1)}\uplus\ldots\uplus\theta^{(m)}:=\big\{(v_{1, l^{(1)}_h}, v_{1,r^{(1)}_h})\big\}_{h=1}^{n_1}\uplus \ldots \uplus \big\{(v_{m,l^{(m)}_h}, v_{m,r^{(m)}_h}) \big\}_{h=1}^{n_m}\notag\\
:=&\big\{(v_{i_k},v_{j_k}):k\in\{1,\ldots,n\}
\text{ and each  }(v_{i_k},v_{j_k})\text{ is a certain }(v_{p,l^{(p)}_h}, v_{p,r^{(p)}_h})\big\}
\end{align*}
Furthermore, for any non--empty subsets $P^{(1)}\subset PP(V_1),\ldots, P^{(m)}\subset PP(V_m)$, one defines
\begin{align*}
P^{(1)}\uplus\ldots\uplus P^{(m)}
:=\big\{\theta^{(1)}\uplus\ldots\uplus \theta^{(m)}:\theta^{(p)}\in P(V_p)\text{ for any } p\in\{1,\ldots,m\}\big\}
\end{align*}
For example, if $V':=\{1,3,4,5\}$ and $V'':=\{2 ,6\}$, along with $\theta':=\{(1,5),$ $(3,4)\}\in PP(V')$ and $\theta'':=\{(2,6)\}\in PP(V'')$, then $\theta' \uplus\theta''=\{(1,5),(3,4),$ $(2,6)\}$. 

\begin{remark}\label{CaCon18g} 1) It is evident that $\theta^{(1)}\uplus\ldots\uplus \theta^{(m)}$ belongs to $PP(V)$ if, for any $p\in\{1, \ldots,m\}$, $\theta^{(p)}:=\big\{ (v_{p,l^{(p)}_h},v_{p, r^{(p)}_h})\big\}_{h=1}^{n_p}\in PP(V_p)$; moreover, ``$\uplus$'' is an associative operation.

2) It is worth noting that for any $P^{(1)}\subset PP(V_1)$ and $P^{(2)}\subset PP(V_2)$, the gluing $P^{(1)}\uplus P^{(2)}$ is contained in $PP(V_1\cup V_2)$; however,

$\bullet$ $P^{(1)}\uplus P^{(2)}$ may not necessarily be a subset of  $NCPP(V_1\cup V_2)$ even if $P^{(p)}\subset NCPP(V_p)$ for both $p\in\{1,2\}$; e.g., in the aforementioned example, both $\theta'$ and $\theta''$ belong to $NCPP(V')$ and $NCPP(V'')$ respectively, whereas the gluing $\theta' \uplus\theta''$ belongs to $PP(V'\cup V'')$ but not to  $NCPP(V'\cup V'')$.

$\bullet$ $PP(V_1)\uplus PP(V_2)$ is not necessarily equal to $PP(V_1\cup V_2)$; e.g., consider the aforementioned case, $\{(1,2),(3,4), (5,6)\}$ belongs to $PP(V_1\cup V_2)$ but not to $PP(V_1)\uplus PP(V_2)$; in fact an element of $PP (V_1\cup V_2)$, says $\{(v_{i_h},v_{j_h})\}_{h=1} ^{n_1+n_2}$, belongs to $PP(V_1)\uplus PP(V_2)$ if and only if for any $h\in\{1,\ldots,n_1+n_2\}$, $v_{i_h}$ and $v_{j_h}$ belong to the {\bf same} $V_p$.
\end{remark}

\begin{proposition}\label{CaCon19} For any $n\in\mathbb{N}^*$ and $V:=\{v_1,\ldots, v_{2n+2}\}$,
for any  $\varepsilon\in \{-1,1\}^V_+$, let's define $m:=\max\{h:\varepsilon (v_h)=-1\}$, $V_k:=V\setminus \{v_m,v_{m+k}\}$ and $\varepsilon_{k}:=$ the restriction of $\varepsilon$ to $V_k$ for any $k\in\{1,\ldots, 2n+2-m\}$. We have the following:
\begin{align}\label{CaCon19x}
\varepsilon_k\in\{-1,1\}^{V_k}_+,\quad\forall k\in\{1,\ldots,2n+2-m\}
\end{align}
and
\begin{align}\label{CaCon19y}
PP(V, \varepsilon)=\bigcup_{k\in\{1,\ldots,2n+2-m\}}
\{(v_m,v_{m+k})\}\uplus PP(V_k, \varepsilon_{k})
\end{align}
Moreover, the sets $\{(v_m,v_{m+k})\}\uplus PP(V_k, \varepsilon_{k})$'s are pairwise disjoint.
\end{proposition}

\begin{proof}Since $(v_m,v_{m+k})$ is distinct from $(v_m,v_{m+j})$ for any different $k,j\in\{1,\ldots, 2n+2-m\}$, we get the disjointness of sets $\{(v_m,v_{m+k})\}\uplus PP(V_k, \varepsilon_{k})$'s.

For any $\varepsilon\in \{-1,1\}^V_+$, for any $k\in\{1,\ldots, 2n+2-m\}$ and $p\in\{1,\ldots,2n+2\}\setminus\{m,m+k\}$, we have:
\begin{align*}
&\sum_{h\in\{p,\ldots,2n+2\}\setminus\{m,m+k\}} \varepsilon_{k}(v_h)=\sum_{h\in\{p,\ldots,2n+2\}\setminus\{m,m+k\}} \varepsilon(v_h)\\
=&\begin{cases}\sum_{h\in\{p,\ldots,2n+2\} \setminus\{m+k\}} 1,&\text{ if }p>m\\ \sum_{h\in\{p,\ldots,2n+2\}}\varepsilon(v_h)
-\varepsilon(v_m)-\varepsilon(v_{m+k}),&\text{ if }p<m
\end{cases}
\end{align*}
which is positive since $\varepsilon(v_m)=-1$ and $\varepsilon(v_{m+k})=1$. Moreover, for $p=1$,
\begin{align*}
&\sum_{h\in\{1,\ldots,2n+2\} \setminus\{m,m+k\}}\varepsilon_{k}(v_h)
=\sum_{h\in\{1,\ldots,2n+2\}}\varepsilon(v_h)
-\varepsilon(v_m)-\varepsilon(v_{m+k})\\
=&\sum_{h\in\{1,\ldots,2n+2\}}\varepsilon(v_h)=0
\end{align*}
Summing up, one gets \eqref{CaCon19x}.

For any $\{(v_{l_h},v_{r_h})\}_{h=1}^{n+1}\in PP(V, \varepsilon)$ with $l_{n+1}=m$, $r_{n+1}$ must be some $m+k$, where $k\in\{1,\ldots,2n+2-m\}$. As a result, $\{(v_{l_h},v_{r_h})\} _{h=1}^{n}\in PP(V_k, \varepsilon_k)$ and so \eqref{CaCon19y} is obtained. \end{proof}

\section{Vacuum expectation of products of creation and annihilation operators and characterization of $\mathcal{P}_n(\varepsilon)$'s} \label{DCTsec(q,m)05}\medskip

For any $n\in\mathbb{N}^*$, $V:=\{v_1,\ldots,v_{2n}\}$ 
and $\varepsilon\in \{-1,1\}^V_+$, we will use the following notations throughout:
\begin{align*}
V_{0,1}^\varepsilon:=&\{v_{l^\varepsilon_h}, v_{r^\varepsilon_h}: d_\varepsilon(l^\varepsilon_h, r^\varepsilon_h)\le 1\}\\
\varepsilon_{0,1} :=& \varepsilon\big\vert_{V_{0,1}^\varepsilon}:= \text{the restriction of $\varepsilon$ to }V_{0,1}^\varepsilon
\end{align*}
and
\begin{align}\label{CaCon15a1}
\mathcal{P}_n(\varepsilon):=\{(v_{l^\varepsilon_h}, v_{r^\varepsilon_h})\}_{h\le n:\, d_\varepsilon (l^\varepsilon_h, r^\varepsilon_h) \ge2}\uplus
PP(V_{0,1}^\varepsilon, \varepsilon_{0,1})
\end{align}
In other words, an arbitrary element of $\mathcal{P}_n(\varepsilon)$ consists of:

$\bullet$ all pairs $(v_{l^\varepsilon_h}, v_{r^ \varepsilon_h})$ with the depth greater or equal to 2;

$\bullet$ an arbitrary element of $PP(V_{0,1}^\varepsilon, \varepsilon_{0,1})$.

\begin{theorem}\label{CaCon15} For any $n\in\mathbb{N} ^*$ and $V:=\{v_1,\ldots,v_{2n}\}$,
for any $\varepsilon\in \{-1,1\}^V_+$, the following statements hold:

1) The set $\mathcal{P}_n(\varepsilon)$ introduced in \eqref{CaCon15a1} contains the counterpart of $\varepsilon$, i.e., $\{(v_{ l^\varepsilon_h}, v_{r^\varepsilon_h})\}_{h=1}^n \in \mathcal{P}_n(\varepsilon)$.

2) $\mathcal{P}_n(\varepsilon)$ is a subset of $PP(2n,\varepsilon)$ such that
\begin{align}\label{CaCon15a0}
A^{\varepsilon(v_1)}(f_1)\ldots A^{\varepsilon(v_{2n})} (f_{2n}) \Phi=&\sum_{ \theta:=\{(v_{l_h},v_{r_h})\} _{h=1}^n \in\mathcal{P}_n (\varepsilon)}q^{c(\theta)} \prod_{h=1}^n\langle f_{l_h},f_{r_h}\rangle \Phi\notag\\
&\text{for all } q\in[-1,1]\text{ and }\{ f_1,\ldots,f_{2n}\} \subset \mathcal{H}
\end{align}
In particular, for any $n\in{\mathbb N}^*$, for any $q\in[-1,1]$ and $f\in \mathcal{H}$,
\begin{align}\label{CaCon15a3}
A^{\varepsilon(1)}(f)\ldots A^{\varepsilon(2n)}(f) \Phi =\Vert f\Vert^{2n}\sum_{ \theta:=\{(v_{l_h} ,v_{r_h})\} _{h=1}^n\in\mathcal{P}_n(\varepsilon)}q^{c(\theta)}\Phi
\end{align}
\end{theorem}

\begin{proof} It follows from the definition that $\mathcal{P}_n(\varepsilon)\subset PP(2n,\varepsilon)$.

\eqref{CaCon15a0} evidently guarantees the validity of \eqref{CaCon15a3} for any $n\in{\mathbb N}^*$, $q\in[-1,1]$ and $f\in \mathcal{H}$. Moreover,  since $\{(v_{l^\varepsilon_h}, v_{r^\varepsilon_h})\}_{h\le n:\, d_\varepsilon (l^\varepsilon_h, r^\varepsilon_h) \le 1}$ belongs to $PP(V_{0,1}^\varepsilon, \varepsilon_{0,1}) $, we have successfully established the validity of the statement 1) as follows:
\begin{align}\label{CaCon15d}
&\{(v_{ l^\varepsilon_h}, v_{r^\varepsilon_h}) \} _{h=1}^n \notag\\
=&\{(v_{l^\varepsilon_h}, v_{r^\varepsilon_h})\}_{h\le n:\, d_\varepsilon (l^\varepsilon_h, r^\varepsilon_h) \ge2}\uplus \{(v_{l^\varepsilon_h}, v_{r^\varepsilon_h})\}_{h\le n:\, d_\varepsilon (l^\varepsilon_h, r^\varepsilon_h) \le1}\in \mathcal{P}_n(\varepsilon)\end{align}

We now proceed to prove \eqref {CaCon15a0}. By virtue of the affirmation 2) of Proposition \ref{DCT28}, the equality in \eqref{CaCon15a0} can be equivalently expressed as:
\begin{align}\label{CaCon15b}
&\big\langle \Phi,A^{\varepsilon(1)}(f_1)\ldots A^{\varepsilon(2n)}(f_{2n})\Phi\big\rangle\notag\\
=&\sum_{ \theta:=\{(v_{l_h},v_{r_h})\}_{h=1}^n \in\mathcal{P}_n (\varepsilon)}q^{c(\theta)} \prod_{h=1}^n\langle f_{l_h},f_{r_h}\rangle
\end{align}
Moreover, without loss of generality, we assume $v_k=k$ for any $k$.

By definition, $\{(l^\varepsilon_h, r^ \varepsilon_h)\} _{h\le n:\, d(l^\varepsilon_h, r^\varepsilon_h)\le 1}$ forms a non--crossing pair partition of $V_{0,1}^\varepsilon$, and so, $\{(l^\varepsilon_h, r^\varepsilon_h)\} _{h\le n:\, d(l^\varepsilon_h, r^\varepsilon_h)\le 1}\in PP(V_{0,1}^\varepsilon,\varepsilon_{0,1})$. This leads to the following simplified version of \eqref{CaCon15d}:
\[\{(l^\varepsilon_h, r^\varepsilon_h)\}_{h\in\{1,
\ldots,n\}}=\{(l^\varepsilon_h, r^\varepsilon_h) \}_ {h\le n:\,d(l^\varepsilon_h,r^\varepsilon_h)\ge2}\uplus \{(l^\varepsilon_h, r^ \varepsilon_h)\} _{h\le n:\,d(l^\varepsilon_h,r^\varepsilon_h)\le 1}\in \mathcal{P}_n(\varepsilon)\]

For any $n\in\{1,2\}$ and $\varepsilon\in\{-1,1\}^{2n}_ +$, it is trivial that $d(l^\varepsilon_h,r^\varepsilon_h)\in\{0,1\}$ for any $h\in\{1,2\}$. Therefore, by the definition, $V^\varepsilon_{0,1}=\{1,\ldots,2n\}$ and $\mathcal{P}_n (\varepsilon)=PP(2n,\varepsilon)$. Consequently,

$\bullet$ for $n=1$, the only element in $\{-1,1\}^{2} _+$ is $\varepsilon=(-1,1)$; thus $\mathcal{P}_n(\varepsilon)=PP(2,\varepsilon)
=\{(1,2)\}$, $c(\{(1,2)\})=0$ and \eqref{CaCon15b} holds:
\[\sum_{\{(l_h,r_h)\}_{h=1}^n\in \mathcal{P}_n (\varepsilon)} q^{c(\{(l_h,r_h)\}_{h=1}^n)}
\prod_{h=1}^n \langle f_{l_h},f_{r_h} \rangle
=\langle f_{1},f_{2}\rangle =\big\langle \Phi,A(f_1) A^{+}(f_{2})\Phi\big\rangle \]

$\bullet$ for $n=2$ and $\varepsilon=(-1,1,-1,1)$, we have $\mathcal{P}_n (\varepsilon)= PP(4,\varepsilon)= \{(1,2),(3,4)\}$,  $c(\{(1,2),(3,4)\})$ $=0$ and \eqref{CaCon15b} holds:
\begin{align*}
&\sum_{\{(l_h,r_h)\}_{h=1}^n\in\mathcal{P}_n (\varepsilon)}  q^{c(\{(l_h,r_h)\}_{h=1}^n)} \prod_{h=1}^n\langle f_{l_h},f_{r_h}\rangle =\big\langle \Phi,A(f_1) A^{+}(f_{2})A(f_{3}) A^{+}(f_{4})\Phi\big\rangle
\end{align*}

$\bullet$ for $n=2$ and $\varepsilon=(-1,-1,1,1)$, $\mathcal{P}_n (\varepsilon)= PP(4,\varepsilon)$ comprises two elements $\theta_1:=\{(1,4),(2,3)\}$ and $\theta_2:=\{(1,3),(2,4)\}$; hence, $c(\theta_1)=0$, $c(\theta_2)=1$ and \eqref{CaCon15b} holds as well:
\begin{align*}
&\sum_{ \theta:=\{(l_h,r_h)\}_{h=1}^n\in\mathcal{P}_n (\varepsilon)} q^{c(\{(l_h,r_h)\}_{h=1}^n)} \prod_{h=1}^n\langle f_{l_h},f_{r_h}\rangle \notag\\
=&\langle f_{1},f_{4}\rangle \cdot \langle f_{2},f_{3}\rangle +q\langle f_{1},f_{3}\rangle \cdot \langle f_{2},f_4\rangle=\big\langle \Phi, A(f_1) A(f_2)A^+(f_3)A^+(f_4) \Phi\big\rangle
\end{align*}
Summing up, \eqref{CaCon15a0} has been demonstrated  for $n=1,2$. Now, assuming it has been proven for $n\ge2$, let's extend the proof to cover the case of $n+1\ge 3$.

For any $\varepsilon\in \{-1,1\}^{2n+2}_+$ and $\theta :=\{(l_h,r_h)\}_{h=1}^{n+1}\in PP(2n+2,\varepsilon)$, it is clear that $\big\{l_h:h\in\{1,\ldots, n+1\}\big\}= \varepsilon^{-1}(\{-1\})$ and $l_1<\ldots<l_{n+1}$, in particular, $l_1=1$ and $n+1\le l_{n+1}=\max\varepsilon ^{-1}(\{-1\})\le 2n+1$. Now let's prove \eqref{CaCon15a0} according to the following three cases:
\[l_{n+1}= 2n+1,\quad l_{n+1}<2n,\quad l_{n+1}= 2n\]

\underline{The $1^{st}$ case: $l_{n+1}= 2n+1$}, In this case, we trivially have that
\begin{align}\label{CaCon15c3} l^{\varepsilon}_{n+1}=2n+1,\quad r^{\varepsilon}_{n+1}= 2n+2,\quad  d_{\theta}(l^{\varepsilon}_{n+1}, r^{\varepsilon}_{n+1})=0
\end{align}
Denoting $\varepsilon'$ as the restriction of $\varepsilon$ to the set $\{1,\ldots, 2n\}$ (i.e. $\varepsilon'(h)=\varepsilon(h)$ for all $h\in\{1, \ldots,2n\}$), and $\theta':=\{(l'_h,r'_h)\}_{h=1}^n$ as its counterpart, we have:
\begin{align}\label{CaCon15c}
l'_h=l^{\varepsilon}_h,\quad r'_h=r^{\varepsilon}_h\ \text{ and }\  d_{\theta'}(l'_h, r'_h)= d_{\theta}(l^{\varepsilon}_h, r^{\varepsilon}_h)\quad \forall h\in\{1,\ldots,n\}
\end{align}
and
\begin{align*}
V^{\varepsilon'}_{0,1}=&\{l'_h, r'_h:\,h\le n\text{ and }d_{\theta'} (l'_h, r'_h)\in\{0,1\}\}\\ \overset{\eqref{CaCon15c}}=&\{l^{\varepsilon}_h, r^{ \varepsilon}_h:\,h\le n\text{ and }d_{\theta} (l^{ \varepsilon}_h, r^{\varepsilon}_h)\in\{0,1\}\}
\end{align*}
It follows, by combining these results and \eqref{CaCon15c3}, that
\begin{align}\label{CaCon15c2}
V^{\varepsilon}_{0,1}&=\{l^{\varepsilon}_h, r^{ \varepsilon}_h:\,h\le n+1\text{ and }d_{\theta} (l^{ \varepsilon}_h, r^{\varepsilon}_h)\in\{0,1\}\}\notag\\
&=\{l^{\varepsilon}_h, r^{ \varepsilon}_h:\,h\le n\text{ and }d_{\theta} (l^{ \varepsilon}_h, r^{\varepsilon} _h)\in\{0,1\}\}\cup\{l^{\varepsilon}_{n+1}, r^{\varepsilon}_{n+1}\}\notag\\
&=V^{\varepsilon'}_{0,1} \cup\{2n+1,2n+2\}
\end{align}
and
\begin{align*}
\varepsilon_{0,1}(j)=\varepsilon\big\vert_
{V^\varepsilon_{0,1}}(j)\overset{\eqref{CaCon15c2}}=
\varepsilon\big\vert_{V^{\varepsilon'}_{0,1}\cup\{2n+1, 2n+2\}}(j)=\begin{cases}\varepsilon'_{0,1}(j),
&\text{ if }j\in V^{\varepsilon'}_{0,1}\\
-1,&\text{ if }j=2n+1\\1,&\text{ if }j=2n+2
\end{cases}
\end{align*}
Consequently,
\begin{align}\label{CaCon15c4a}
PP(V_{0,1}^{\varepsilon}, \varepsilon_{0,1})&= PP(V_{0,1}^{\varepsilon'}, \varepsilon'_{0,1}) \uplus \{(2n+1,2n+2)\}\notag\\
&=PP(V_{0,1}^{\varepsilon'}, \varepsilon'_{0,1}) \uplus \{(l^{\varepsilon}_{n+1},r^{\varepsilon}_{n+1})\}
\end{align}
With the assistance of \eqref{CaCon15c} and our inductive assumption, we can conclude that
\begin{align}\label{CaCon15c5}
&\big\langle \Phi,A^{\varepsilon(1)}(f_1)\ldots A^{\varepsilon(2n+2)}(f_{2n+2})\Phi\big\rangle\notag\\
=&\big\langle \Phi,A^{\varepsilon(1)}(f_1)\ldots A^{\varepsilon(2n)}(f_{2n})A(f_{2n+1})A^+(f_{2n+2})
\Phi\big\rangle \notag\\
=&\langle f_{2n+1},f_{2n+2}\rangle \cdot
\big\langle \Phi,A^{\varepsilon(1)}(f_1)\ldots A^{\varepsilon(2n)}(f_{2n})\Phi\big\rangle \notag\\
=&\langle f_{2n+1},f_{2n+2}\rangle  \cdot
\big\langle \Phi,A^{\varepsilon'(1)}(f_1)\ldots A^{\varepsilon'(2n)}(f_{2n})\Phi\big\rangle\notag\\
=&\langle f_{2n+1},f_{2n+2}\rangle\cdot\sum_{\theta':= \{(l_h, r_h)\} _{h=1}^n\in\mathcal{P}_n(\varepsilon')} q^{c(\theta')}\prod_{h=1}^n\langle f_{l_h},f_{r_h}\rangle
\end{align}
where,
\begin{align*}
\mathcal{P}_n (\varepsilon')=\{(l_h, r_h)\}_{h\le n,\, d_{\varepsilon'}(l_h,r_h)\ge2}\uplus PP(V_{0,1}^{\varepsilon'}, \varepsilon'_{0,1})
\end{align*}
Notice that

$\bullet$ \eqref{CaCon15c3} ensures that $\{(l_h,r_h)\}_{h=1}^{n} \uplus \{(2n+1,2n+2)\}$ runs over $\mathcal{P}_n (\varepsilon')\uplus \{(2n+1,2n+2)\}$ as $\{(l_h,r_h)\}_{h=1}^{n}$ running over $\mathcal{P}_n (\varepsilon')$;

$\bullet$ the first equality in \eqref{CaCon17c} gives us $c(\theta')=c(\theta)$.\\
We know that \eqref{CaCon15c5} becomes to
\begin{align}\label{CaCon15c5a}
&\big\langle \Phi,A^{\varepsilon(1)}(f_1)\ldots A^{\varepsilon(2n+2)}(f_{2n+2})\Phi\big\rangle\notag\\
=&\sum_{ \theta:=\{(l_h,r_h)\}_{h=1}^{n+1}\in{\mathcal P}_n (\varepsilon')\uplus \{(2n+1,2n+2)\} }q^{c(\theta)} \prod_{h=1}^{n+1}\langle f_{l_h},f_{r_h}\rangle
\end{align}
Moreover, \eqref{CaCon15c} and \eqref{CaCon15c4a} guarantee that
\begin{align*}
&\mathcal{P}_n (\varepsilon')\uplus \{2n+1,2n+2\}= \mathcal{P}_n (\varepsilon')\uplus \{(l^{\varepsilon} _{n+1},r^{\varepsilon}_{n+1})\} \notag\\
=&\{(l'_h, r'_h)\}_{h\le n,\,d_{\varepsilon'} (l'_h,r'_h)\ge2} \uplus PP(V_{0,1}^{\varepsilon'}, \varepsilon' _{0,1}) \uplus \{(l^{\varepsilon}_{n+1}, r^{\varepsilon}_{n+1})\}\notag\\
=&\{(l_h, r_h)\}_{h\le n,\, d_{\varepsilon}(l_h,r_h) \ge2} \uplus PP(V_{0,1}^{\varepsilon},\varepsilon_{0,1})
\overset{\eqref{CaCon15a1}}=\mathcal{P}_{n+1} (\varepsilon)
\end{align*}
In other words, \eqref{CaCon15c5a} is the same as the following equality:
\begin{align}\label{CaCon15c8}
&\big\langle \Phi,A^{\varepsilon(1)}(f_1)\ldots A^{\varepsilon(2n+2)}(f_{2n+2})\Phi\big\rangle\notag\\
=&\sum_{ \theta:=\{(l_h,r_h)\}_{h=1}^{n+1}\in\mathcal{P} _{n+1} (\varepsilon)}  q^{c(\theta)}\prod_{h=1}^{n+1} \langle f_{l_h},f_{r_h}\rangle
\end{align}

\underline{The $2^{nd}$ case: $l_{n+1}< 2n$}, In this case, it is easy to observe that:
\begin{align}\label{CaCon15e} l^{\varepsilon}_{n+1}<2n, \quad r^{\varepsilon}_{n+1}=l^{\varepsilon}_{n+1}+1\le 2n ,\quad d_{\theta}(l^{\varepsilon}_{n+1}, r^{\varepsilon}_{n+1})\ge 2
\end{align}
Moreover,  by denoting

$\bullet$ $\varepsilon'$ as the restriction of $\varepsilon$ to the set $\{1,\ldots,2n+2\}\setminus \{l^{\varepsilon}_{n+1}, r^{\varepsilon}_{n+1}\}$, which is $\{1,\ldots,2n+2\} \setminus\{l^{\varepsilon} _{n+1}, l^{\varepsilon}_{n+1} +1\}$;

$\bullet$ $\theta':=\{(l'_h,r'_h)\}_{h=1}^{n}$ as the counterpart of $\varepsilon'$;\\
the three equalities in \eqref{CaCon15c} remain valid for any $h\in\{1,\ldots,n\}$. Consequently,
\begin{align*}
V^{\varepsilon'}_{0,1}&=\{l'_h, r'_h:\,h\le n\text{ and }d_{\theta'} (l'_h, r'_h)\in\{0,1\}\}\notag\\
&\overset{\eqref{CaCon15c}}=\{l^{\varepsilon}_h, r^{\varepsilon}_h:\,h\le n\text{ and }d_{\theta} (l^{ \varepsilon}_h, r^{\varepsilon}_h)\in\{0,1\}\}\notag\\
&=\{l^{ \varepsilon}_h, r^{ \varepsilon}_h:\,h\le n+1\text{ and }d_{\theta} (l^{ \varepsilon}_h, r^{\varepsilon} _h)\in\{0,1\}\} =V^{\varepsilon}_{0,1}
\end{align*}
where the penultimate equality is guaranteed by the last inequality in \eqref{CaCon15e}. Therefore,
\begin{align*}
\text{Dom}(\varepsilon_{0,1})=V^{\varepsilon}_{0,1}=V^{\varepsilon'}_{0,1}=\text{Dom}(\varepsilon'_{0,1}),\quad \varepsilon_{0,1}(j)=\varepsilon'_{0,1}(j),\quad \forall j\in V^{\varepsilon}_{0,1}
\end{align*}
and so
\begin{align}\label{CaCon15e4b}
\varepsilon_{0,1}=\varepsilon'_{0,1},\quad\ PP(V_{0,1}^{\varepsilon}, \varepsilon_{0,1})&= PP(V_{0,1}^{\varepsilon'}, \varepsilon'_{0,1})
\end{align}

With the help of \eqref{DCT05g1}, \eqref{CaCon15c} and the assumption of induction, we have:
\begin{align*}
&\big\langle \Phi,A^{\varepsilon(1)}(f_1)\ldots A^{\varepsilon(2n+2)}(f_{2n+2})\Phi\big\rangle\notag\\
=&\big\langle \Phi,A^{\varepsilon(1)}(f_1)\ldots A^{\varepsilon(l^{\varepsilon}_{n+1}-1)}(f_{l^{\varepsilon}_{n+1}-1})A(f_{l^{\varepsilon}_{n+1}})A^+(f_{l^{ \varepsilon}_{n+1}+1})\ldots A^+(f_{2n+2})\Phi \big\rangle \notag\\
=&\langle f_{l^{\varepsilon}_{n+1}}, f_{r^{\varepsilon} _{n+1}}\rangle \cdot\big\langle \Phi,A^{\varepsilon(1)} (f_1)\ldots A^{\varepsilon(l^{\varepsilon}_{n+1}-1)} (f_{l^{\varepsilon}_{n+1}-1})A^+(f_{r^{\varepsilon}_{n+1}+1})\ldots A^+(f_{2n+2})\Phi\big\rangle \notag\\
=&\langle f_{l^{\varepsilon}_{n+1}}, f_{r^{\varepsilon}_{n+1}}\rangle  \cdot
\big\langle \Phi,A^{\varepsilon'(1)}(f_1)\ldots
A^{\varepsilon'(l^{\varepsilon}_{n+1}-1)}(f_{l^{\varepsilon}_{n+1}-1})\notag\\
&\hspace{3.4cm}A^{\varepsilon'(l^{\varepsilon}_{n+1}+1)}
(f_{r^{\varepsilon}_{n+1}+1})\ldots A^{\varepsilon'(2n+2)}(f_{2n+2})\Phi\big\rangle\notag\\
=&\langle f_{l^{\varepsilon}_{n+1}}, f_{r^{\varepsilon} _{n+1}}\rangle \cdot\sum_{\theta':=\{(l_h,r_h)\} _{h=1}^n\in\mathcal{P}_n (\varepsilon')}q^{c(\theta')} \prod_{h=1}^n\langle f_{l_h},f_{r_h}\rangle
\end{align*}
where,
\begin{align}\label{CaCon15f1}
\mathcal{P}_n (\varepsilon')=&\{(l^\varepsilon_h, r^\varepsilon_h)\}_{h\le n,\, d_{\varepsilon'} (l^\varepsilon_h,r^\varepsilon_h)\ge2}\uplus PP(V_{0,1}^{\varepsilon'}, \varepsilon'_{0,1})\notag\\
\overset{\eqref{CaCon15e4b}}=&\{(l^\varepsilon_h, r^\varepsilon_h)\}_{h\le n,\, d_{\varepsilon} (l^\varepsilon_h,r^\varepsilon_h)\ge2}
\uplus PP(V_{0,1}^{\varepsilon}, \varepsilon_{0,1})
\end{align}
By noticing that

$\bullet$ \eqref{CaCon15f1} and \eqref{CaCon15e} guarantee that, as $\{(l_h,r_h)\}_{h=1}^{n}$ running over $\mathcal{P}_n (\varepsilon')$, the gluing $\{(l_h,r_h)\}_{h=1}^{n} \uplus \{(l^{\varepsilon}_{n+1},r^{\varepsilon}_ {n+1})\}$ runs over the set $\{(l_h, r_h)\}_{h\le n+1,\, d_{\varepsilon} (l_h,r_h)\ge2}$ $\uplus PP(V_{0,1}^{\varepsilon},\varepsilon _{0,1})$ (i.e. $\mathcal{P}_{n+1}(\varepsilon)$ by definition);

$\bullet$ the first equality in \eqref{CaCon17c} makes sure that $c(\theta')=c(\theta)$.\\
We can once again derive \eqref{CaCon15c8} by using  \eqref{CaCon15c5}.

\underline{The $3^{rd}$ case: $l_{n+1}=2n$}, In this case, the following equalities are trivial:
\begin{align}\label{CaCon15g} l^{\varepsilon}_{n+1}=\max\varepsilon^{-1}(\{-1\})=2n,
\quad r^{\varepsilon}_{n+1}= 2n+1 ,\quad d_{\theta} (l^{\varepsilon}_{n+1}, r^{\varepsilon}_{n+1})=1
\end{align}
Moreover, there exists a unique $h_0\in\{1,\ldots,n\}$ such that
\begin{align}\label{CaCon15g0}
r^{\varepsilon}_{h_0}= 2n+2 ,\quad d_{\theta} (l^{\varepsilon}_{h_0}, r^{\varepsilon}_{h_0})=0
\end{align}
Therefore,
\begin{align}\label{CaCon15g10}
V_{0,1}^{\varepsilon}=&\{l^\varepsilon_h,r^\varepsilon_h:\, h\le n+1,\, d_\theta(l^\varepsilon_h, r^\varepsilon _h)\le 1\}\notag\\
=&\{l^\varepsilon_h, r^\varepsilon_h:\, h_0\ne h\in \{1,\ldots,n\},\, d_\theta(l^\varepsilon_h, r^\varepsilon_h)\le 1\}\notag\\
\cup &\{l^\varepsilon_{h_0}, 2n,2n+1,2n+2\}
\end{align}
where, the second equality is based on \eqref{CaCon15g} and \eqref{CaCon15g0}. By denoting, for any $k\in\{1,2\}$,
\begin{align}\label{CaCon15g9b}
V^{(k)}:=&\{1,\ldots,2n+2\}\setminus\{2n,2n+k\}\notag\\ \varepsilon^{(k)}:=&\varepsilon \big\vert_{V^{(k)}}= (\varepsilon(1),\ldots,\varepsilon(2n-1),1)
\end{align}
Proposition \ref{CaCon19} guarantees that
\begin{align*}
\varepsilon^{(k)}&\in \{-1,1\}^{V^{(k)}}_+\notag\\ \theta^{(k)}&:=\{(l^{(k)}_h,r^{(k)}_h)\}_{h=1}^{n}
:=\text {the counterpart of }\varepsilon^{(k)}
\in PP(V^{(k)})
\end{align*}
Moreover, we find the following analogues of \eqref{CaCon15c} (recall that exists a unique $h_0 \in\{1,\ldots,n\}$ such that $r^{\varepsilon}_{h_0} = 2n+2$): for any $k\in\{1,2\}$ and $h\in\{1,\ldots,n\}$,
\begin{align}\label{CaCon15g6}
&l^{(k)}_h=l^{\varepsilon}_h, \quad  r^{(k)}_h=\begin{cases} r^{\varepsilon}_h,&\text{ if }h\ne h_0\\ 2n+2,&\text{ if }h= h_0\text{ and }k=1\\ 2n+1,&\text{ if }h= h_0\text{ and }k=2
\end{cases}\notag\\
&d_{\theta^{(k)}}(l^{(k)}_h, r^{(k)}_h)=  d_{\theta}(l^{\varepsilon}_h, r^{\varepsilon}_h)
\end{align}
Now we introduce analogues of $V_{0,1}^\varepsilon$ and $\varepsilon_{0,1}$ as follows:
\begin{align*}
V_{0,1}^{\varepsilon,k}:=&\big(V^{(k)}\big)_{0,1}^
\varepsilon:=\{l^{(k)}_h,r^{(k)}_h:\,h\le n,\ d_{\theta^{(k)}}(l^{(k)}_h, r^{(k)}_h)\le 1\}\notag\\ \varepsilon_{0,1}^{(k)}:=&\big(\varepsilon^{(k)}\big)
_{0,1}:=\varepsilon^{(k)}\big\vert_{V_{0,1}^{(\varepsilon,k)}} ,\quad \forall k=1,2
\end{align*}
For any $k\in\{1,2\}$, we can conclude the following:

$\bullet$ \eqref{CaCon15g6} and the fact that $d_{\theta^{(k)}} (l^{(k)}_{h_0}, r^{(k)}_{h_0})=d_{\theta^{(k)}} (l^\varepsilon_{h_0}, r^{(k)}_{h_0})=0$ tell us
\begin{align}\label{CaCon15g5c}
V_{0,1}^{\varepsilon,k}=&\{l^{\varepsilon}_h,
r^{\varepsilon}_h:\,h_0\ne h\in\{1,\ldots,n\},\ d_{\theta}(l^{\varepsilon}_h, r^{\varepsilon}_h)\le 1\}
\cup \{l^\varepsilon_{h_0}, r^{(k)}_{h_0}\}\notag\\
=&\{l^{\varepsilon}_h,
r^{\varepsilon}_h:\,h_0\ne h\in\{1,\ldots,n\},\ d_{\theta}(l^{\varepsilon}_h, r^{\varepsilon}_h)\le 1\}\notag\\
\cup&\begin{cases}
\{l^\varepsilon_{h_0}, 2n+2\},&\text{ if }k=1\\
\{l^\varepsilon_{h_0}, 2n+1\},&\text{ if }k=2
\end{cases}
\end{align}

$\bullet$ Proposition \ref{CaCon19} implies that
$\varepsilon_{0,1}^{(k)}\in \{-1,1\}^{V_{0,1}^{\varepsilon,k}}_+$.

$\bullet$ The fact that $V_{0,1}^{\varepsilon} \supset V_{0,1}^{\varepsilon,k} \subset V^{(k)}$ gives us
\begin{align*}
\varepsilon_{0,1}^{(k)}:=&\varepsilon^{(k)}\big\vert_ {V_{0,1}^{\varepsilon,k}}\overset{\eqref{CaCon15g9b}}
=\Big(\varepsilon\big\vert_{V^{(k)}}\Big)
\Big\vert_{V_{0,1}^{\varepsilon,k}}
=\varepsilon\big\vert_{V^{(k)}\cap V_{0,1}^{\varepsilon,k}}\notag\\
=&\varepsilon\big\vert_{V_{0,1}^{\varepsilon,k}}=
\varepsilon\big\vert_{V_{0,1}^{\varepsilon}\cap V_{0,1}^{\varepsilon,k}}=
\Big(\varepsilon\big\vert_{V^\varepsilon_{0,1}}\Big)
\Big\vert_{V_{0,1}^{\varepsilon,k}}= \varepsilon_{0,1}\big\vert_{V_{0,1}^{\varepsilon,k}}
\end{align*}

$\bullet$ It is obtained, by comparing \eqref{CaCon15g5c} and \eqref{CaCon15g10}, that
\begin{align}\label{CaCon15g8}
V_{0,1}^{\varepsilon,k}= V_{0,1}^{\varepsilon} \setminus\{2n,2n+k\}=V_{0,1}^{\varepsilon}\setminus\{l ^\varepsilon_{n+1},l^{\varepsilon}_{n+1}+k\}
\end{align}

$\bullet$ By using \eqref{CaCon15g8}, the facts that $V_{0,1}^{\varepsilon,k}\cap \{2n,2n+k\}=\emptyset$ and
$\{2n,2n+k\}\subset V_{0,1}^{\varepsilon}$, we get:
\begin{align}\label{CaCon15g11}
V^{\varepsilon,1}_{0,1} \cup\{2n,2n+1\}=
V^{\varepsilon}_{0,1}= V^{\varepsilon,2}_{0,1} \cup\{2n,2n+2\}
\end{align}

With the help of \eqref{CaCon15g11}, Proposition \ref{CaCon19} gives us
\begin{align}\label{CaCon15g5}
PP(V_{0,1}^{\varepsilon}, \varepsilon_{0,1})&= \bigcup_{k=1,2}\{(l^{\varepsilon}_{n+1},l^{\varepsilon} _{n+1}+k)\}\uplus PP(V_{0,1}^{\varepsilon,k}, \varepsilon^{(k)}_{0,1})\notag\\
&=\bigcup_{k=1,2}\{(2n, 2n+k)\}\uplus PP(V_{0,1}^{\varepsilon,k}, \varepsilon^{(k)}_{0,1})
\end{align}
Where, an element of $\{(2n, 2n+1)\}\uplus PP(V_{0,1}^{\varepsilon,1}, \varepsilon^{(1)}_{0,1})$ cannot be the same as an element of $\{(2n, 2n+2)\}\uplus PP(V_{0,1}^{\varepsilon,2}, \varepsilon^{(2)}_{0,1})$ because the pair in the above two pair partitions with the same left index $l^\varepsilon_{n+1}=2n$ are $(2n, 2n+2)$ and $(2n, 2n+1)$ respectively. Therefore, the union in \eqref{CaCon15g5} consists of two disjoint sets.

By the definition, for any $\varepsilon\in \{-1,1\}^{2n+2}_+$, we have:
\begin{align}\label{CaCon15h}
&\mathcal{P}_{n+1}(\varepsilon):=\{(l^\varepsilon_h, r^ \varepsilon_h)\}_{h\le n+1:\,d_\varepsilon (l^\varepsilon_h, r^\varepsilon_h)\ge2} \uplus PP(V_{0,1}^\varepsilon, \varepsilon_{0,1}) \notag\\
\overset{\eqref{CaCon15g}}=&\{(l^\varepsilon_h, r^ \varepsilon_h)\}_{h\le n:\, d_\varepsilon (l^\varepsilon_h, r^\varepsilon_h)\ge2} \uplus PP(V_{0,1}^\varepsilon, \varepsilon_{0,1})\notag\\
\overset{\eqref{CaCon15g5}}=&\bigcup_{k=1,2}
\{(l^\varepsilon_h, r^ \varepsilon_h)\}_{h\le n:\, d_\varepsilon(l^\varepsilon_h, r^\varepsilon_h)\ge2} \uplus PP(V_{0,1}^{\varepsilon,k}, \varepsilon_{0,1} ^{(k)})\uplus \{((2n, 2n+k))\}\notag\\
=&\bigcup_{k=1,2}{\mathcal P}_n(\varepsilon^{(k)})
\uplus \{((2n, 2n+k))\}
\end{align}
Because $l^{\varepsilon}_{n+1}=\max\varepsilon^{-1} (\{-1\})=2n$, \eqref{DCT05g1} shows that
\begin{align*}
&A^{\varepsilon(2n)}(f_{2n})A^{\varepsilon(2n+1)} (f_{2n+1})A^{\varepsilon(2n+2)}(f_{2n+2})\Phi\notag\\
=&A(f_{2n})A^+(f_{2n+1})A^+(f_{2n+2})\Phi\notag\\
=&\langle f_{2n}, f_{2n+1}\rangle A^+(f_{2n+2})\Phi
+q\langle f_{2n}, f_{2n+2}\rangle A^+(f_{2n+1})\Phi
\end{align*}
Thus, utilizing the inductive assumption along with \eqref{CaCon15g6}, \eqref{CaCon15g8} and \eqref{CaCon15g11}, we find that
\begin{align}\label{CaCon15g2}
&\big\langle \Phi,A^{\varepsilon(1)}(f_1)\ldots A^{\varepsilon(2n+2)}(f_{2n+2})\Phi\big\rangle\notag\\
=&\langle f_{2n}, f_{2n+1}\rangle \cdot\big\langle \Phi,A^{\varepsilon(1)}(f_1)\ldots A^{\varepsilon(2n-1)} (f_{2n-1}) A^{\varepsilon(2n+2)}(f_{2n+2})\Phi \big\rangle \notag\\
+&q\langle f_{2n}, f_{2n+2}\rangle \cdot\big\langle \Phi,A^{\varepsilon(1)}(f_1)\ldots A^{\varepsilon(2n-1)} (f_{2n-1}) A^{\varepsilon(2n+1)}(f_{2n+1})\Phi\big\rangle\notag\\
=&\langle f_{2n}, f_{2n+1}\rangle \cdot \sum_{\theta^{(1)}:=\{(l^{(1)}_h,r^{(1)}_h)\}_{h=1}^n\in {\mathcal P}_n(\varepsilon^{(1)})}  q^{c(\theta^{(1)})}\cdot\prod_{h=1}^n\langle f_{l^{(1)}_h}, f_{r^{(1)}_h}\rangle\notag\\
+&q\langle f_{2n}, f_{2n+2}\rangle \cdot
\sum_{\theta^{(2)}:=\{(l^{(2)}_h,r^{(2)}_h)\}_{h=1}^n
\in{\mathcal P}_n(\varepsilon^{(2)})}  q^{c(\theta^{(2)})}\cdot\prod_{h=1}^n\langle f_{l^{(2)}_h}, f_{r^{(2)}_h}\rangle
\end{align}
By denoting $\theta_k:=\theta^{(k)}\uplus \{(2n,2n+k)\}$ for any $k\in\{1,2\}$, we get:

$\bullet$ $\theta_k\in {\mathcal P}_n(\varepsilon^{(k)}) \uplus \{(2n,2n+k)\}\subset PP(\{1,\ldots,2n\})$;

$\bullet$ $\theta_k$ runs over ${\mathcal P}_n(\varepsilon ^{(k)})\uplus \{(2n,2n+k)\}$ as $\theta^{(k)}$ running over ${\mathcal P}_n (\varepsilon^{(k)})$;

$\bullet$ as a consequence of \eqref{CaCon17c},
$c(\theta_1)=c(\theta^{(1)})$ and $c(\theta_2)=c(\theta^{(2)})+1$.\\
So, for any $k\in\{1,2\}$, by denoting  $(l^{(k)}_{n+1},r^{(k)}_{n+1}):= (2n,2n+k)$, we have
$\theta_k=\{(l^{(k)}_h,r^{(k)}_h)\}_{h=1}^{n+1}$ whenever $\theta^{(k)}=\{(l^{(k)}_h,r^{(k)}_h)\} _{h=1}^n$, and moreover, \eqref{CaCon15g2} becomes to
\begin{align*}
&\big\langle \Phi,A^{\varepsilon(1)}(f_1)\ldots A^{\varepsilon(2n+2)}(f_{2n+2})\Phi\big\rangle\notag\\
=&\sum_{\theta_1:=\{(l^{(1)}_h,r^{(1)}_h)\}_{h=1}^{n+1}
\in {\mathcal P}_n(\varepsilon^{(1)}) \uplus\{(2n,2n+1)\}}  q^{c(\theta_1)}\prod_{h=1}^{n+1}
\langle f_{l^{(1)}_h}, f_{r^{(1)}_h}\rangle\notag\\
+&\sum_{\theta_2:=\{(l^{(2)}_h,r^{(2)}_h)\}_{h=1}^{n+1}
\in{\mathcal P}_n(\varepsilon^{(2)}) \uplus\{(2n,2n+2)\}}  q^{c(\theta_2)}\prod_{h=1}^{n+1} \langle f_{l^{(2)}_h}, f_{r^{(2)}_h}\rangle
\end{align*}
which is exactly, thanks to \eqref{CaCon15h},
$\sum_{ \theta:=\{(l_h,r_h)\}_{h=1}^{n+1}\in\mathcal{P} _{n+1}(\varepsilon)}q^{c(\theta)}\prod_{h=1}^{n+1} \langle f_{l_h}, f_{r_h}\rangle $.
\end{proof}

Our final task is to see the uniqueness of ${\mathcal P}_{n} (\varepsilon)$'s.

\begin{theorem}\label{CaCon16} For any $n\in\mathbb{N} ^*$ and $V_n:=\{v_1,\ldots,v_{2n}\}$, for any $\varepsilon\in \{-1,1\}^{V_n}_+$, ${\mathcal P}_{n} (\varepsilon)$ is unique subset of $PP(V_n,\varepsilon)$ verifies \eqref{CaCon15a0}. More precisely, let ${\mathcal Q}_{n} (\varepsilon)$ be a subset of $PP(V_n,\varepsilon)$ for any $n\in\mathbb{N}^*$ and $\varepsilon\in \{-1,1\}^V_+$. If holds the following analogue of \eqref{CaCon15a0}:
\begin{align}\label{CaCon15k}
A^{\varepsilon(v_1)}(f_1)\ldots A^{\varepsilon(v_{2n})}& (f_{2n}) \Phi=\sum_{ \theta:=\{(v_{l^\varepsilon_h}, v_{r^\varepsilon_h})\}_{h=1}^n \in{\mathcal Q}_n (\varepsilon)}q^{c(\theta)} \prod_{h=1}^n\langle f_{l_h},f_{r_h}\rangle \Phi\notag\\
&\text{for all } n\in{\mathbb N}^*,\,q\in[-1,1]\text{ and }\{ f_1,\ldots,f_{2n}\} \subset \mathcal{H}
\end{align}
Then
\begin{align}\label{CaCon16a}{\mathcal Q}_{n} (\varepsilon)={\mathcal P}_{n} (\varepsilon),\qquad\forall n\in{\mathbb N}^*\text{ and }\varepsilon\in \{-1,1\}^{V_n}_+
\end{align}
\end{theorem}
\begin{proof} Without loss of generality, we can assume $v_k=k$ for all $k$, and the equality in \eqref{CaCon15k} can be replaced by:
\begin{align}\label{CaCon15k0}
\big\langle\Phi,A^{\varepsilon(1)}(f_1)\ldots A^{\varepsilon({2n})} (f_{2n}) \Phi\big\rangle =
\sum_{\theta:=\{(l^\varepsilon_h, r^\varepsilon_h) \} _{h=1}^n \in{\mathcal Q}_n (\varepsilon)}q^{c(\theta)} \prod_{h=1}^n\langle f_{l_h},f_{r_h}\rangle
\end{align}

We will now demonstrate \eqref{CaCon16a} through an inductive proof.

For the case of $n=1$, there is a unique element in $\{-1,1\}^{2}_+$, namely $\varepsilon=(-1,1)$, and $PP(2,\varepsilon)=\{(1,2)\}$ (i.e. $l_1=1,\,r_1=2$). Thus, \eqref{CaCon15k0}  simplifies  to
\[
\big\langle\Phi,A(f_1) A^+(f_2) \Phi\big\rangle =
\langle f_1,f_2 \rangle
\]
Moreover, any ${\mathcal P}\subset PP(2,\varepsilon)$ must be either empty or $PP(2,\varepsilon)$. However, ${\mathcal P}$ cannot be empty because $\langle f_1,f_2 \rangle =\sum_{\{(l,r)\}\in {\mathcal P}}\langle f_l,f_r \rangle$ for all $f_1,f_2\in{\mathcal H}$. In other words, the equality ${\mathcal Q}_n (\varepsilon)={\mathcal P}_n (\varepsilon)=PP(2,\varepsilon)$ has been verified for $n=1$.

If the equality in \eqref{CaCon15k0} is proven for $n=m$, then, as argued in the proof of Theorem \ref{CaCon15}, for any $\varepsilon\in\{-1,1\} ^{2m+2}_+$, the validity of \eqref{CaCon15k0} for $n=m+1$ requests that the ${\mathcal Q}_{m+1} (\varepsilon)$ must be:

$\bullet$ ${\mathcal Q}_m(\varepsilon')\uplus\{ (l^{\varepsilon}_{m+1}, r^{\varepsilon}_{m+1})\}$ (i.e., ${\mathcal P}_m(\varepsilon')\uplus\{ (l^{\varepsilon}_{m+1}, r^{\varepsilon}_{m+1})\}$ thanks to the inductive assumption) when $l^{\varepsilon}_{m+1} \ne 2m$. Here, $\varepsilon'$ is the restriction of $\varepsilon$ to the set $\{1,\ldots,2m+2\}\setminus \{l^{\varepsilon}_{m+1},r^{\varepsilon}_{m+1}\}$;

$\bullet$ $\bigcup_{k=1,2} {\mathcal Q}_m(\varepsilon_k) \uplus \{(2m,2m+k)\}$ (i.e., $\bigcup_{k=1,2}{\mathcal P}_m(\varepsilon_k) \uplus \{(2m,2m+k)\}$ thanks to the inductive assumption) when $l^{\varepsilon}_{m+1}= 2m$. Here, for any $k\in\{1,2\}$, $\varepsilon_k$ is the restriction of $\varepsilon$ to the set $\{1,\ldots,2m+2\}\setminus \{2m,2m+k\}$.

Summing up, we have obtained the equality in  \eqref{CaCon16a} for $n=m+1$.\end{proof}

\end{document}